\documentclass[]{interact}

\usepackage{epstopdf}% To incorporate .eps illustrations using PDFLaTeX, etc.
\usepackage{subfigure}% Support for small, `sub' figures and tables

\usepackage[natbibapa,nodoi]{apacite}
\usepackage{graphics}
\usepackage{graphicx}      % include this line if your document contains figures
\usepackage{epsfig}

\usepackage[longnamesfirst,sort]{natbib}% Citation support using natbib.sty
\bibpunct[, ]{(}{)}{;}{a}{,}{,}% Citation support using natbib.sty
\renewcommand\bibfont{\fontsize{10}{12}\selectfont}% To set the list of references in 10 point font using natbib.sty
\usepackage{color}

\theoremstyle{plain}% Theorem-like structures provided by amsthm.sty
\newtheorem{theorem}{Theorem}[section]
\newtheorem{lemma}[theorem]{Lemma}

\theoremstyle{definition}
\newtheorem{definition}[theorem]{Definition}

\theoremstyle{remark}
\newtheorem{remark}{Remark}

\newcommand{\crb}{\color{black}}

\begin{document}

\title{Constrained Optimal Consensus in Multi-agent Systems
with First and Second Order Dynamics}

\author{
\name{Amir Adibzadeh\textsuperscript{a} \thanks{CONTACT Amir A. Suratgar Email: a-suratgar@aut.ac.ir}, Amir A. Suratgar\textsuperscript{a},  Mohammad B. Menhaj\textsuperscript{a}, Mohsen Zamani\textsuperscript{b} }
    \affil{\textsuperscript{a}Electrical Engineering
    Department, Amirkabir University of Technology, Tehran 15914, Iran; \textsuperscript{b}School of Electrical Engineering and Computer Science,
University of Newcastle, Callaghan, NSW 2308, Australia}
}
\maketitle

\begin{abstract}
This paper fully studies distributed optimal consensus
problem in undirected
dynamical networks. We consider a group of networked agents
that are supposed to rendezvous at the optimal point of a collective convex
objective function. Each agent has no knowledge
about the global objective function and only has
access to its own local objective function, {which is a portion of the global one,} and states information
of agents within its neighborhood set.
In this setup, all agents coordinate
with their neighbors to seek  the consensus point that minimizes the network's global objective
function. In the current paper,  we consider  agents with
single-integrator and double-integrator dynamics. We further suppose that agents' movements are limited
by some convex inequality constraints.  In order to find the optimal consensus point under the described scenario,
we combine the interior-point optimization algorithm with a consensus protocol and propose a distributed control law.
The associated convergence analysis based on Lyapunov stability analysis is provided.

\end{abstract}


\section{Introduction}\label{sec1}

In the past, consensus problems in a network of autonomous
agents have been investigated from different aspects \textcolor{black}{such as} communication
topology, agents' dynamics, and the consensus value properties
\cite{olfati2004consensus,ren2007distributed,cheng2016reaching,zhang2012adaptive,wieland2011internal,fan2014semi,rezaee2015average}.
%Constraints on information flow in the network and the lack of information processing center are the main reasons
%why consensus algorithms are often distributed in their nature. Distributed
%consensus means that some or the whole states of all agents converge
%to an agreed value of common interest in a cooperative fashion.
Moreover, in many practical scenarios, the consensus problem under some local
constraints on the agents' states is considered \cite{nedic2010constrained,lin2014constrained,lee2011constrained}.
\cite{lee2011constrained} applied a logarithmic barrier function to guarantee that agents agree
on a consensus value that must belong to the intersection of distinct
convex sets through sharing an auxiliary variable associated with a convex
function representing the constraint set.
To solve set-constrained consensus problems, a distributed consensus protocol was
proposed in \cite{nedic2010constrained}. In this reference,
a consensus protocol is combined with a projection operator,
adopted to satisfy set constraints, in order to move agents to an agreed point
that is restricted to lie in the intersection of local convex constraint sets.
The article \cite{lin2014constrained}
extended the work of \cite{nedic2010constrained} to study the problem
of constrained consensus in unbalanced networks.

In another stream of research, distributed convex optimization problems in a network of agents
are considered. In such problems, each agent is assigned with a local objective function,
and the final consensus value is required to minimize the
sum of individual uncoupled convex functions.
%Many results were obtained based on
%gradients or subgradients of the local cost functions, combined with consensus laws.
The goal is to propose a distributed control law that
achieves a consensus on the minimizer  of the sum of all individual cost functions.
%For instance, some sophisticated dynamics whose stable point
%is supposed to be the solution of the optimization problem was imposed on each agent.
{\cite{nedic2009distributed}} exploited a subgradient-based distributed
method to find an approximate optimal solution to a convex
optimization problem over a network.
{In \cite{lu2012zero}, through an invariant zero-gradient-sum manifold, the states
 of a proposed weight-balanced directed network are driven toward the optimal solution
of an unconstrained convex distributed optimization problem.}

To deal with distributed optimization problems with inequality and
equality constraints, some researches were conducted based on primal-dual methods with continuous-time
agents. \cite{raffard2004distributed} used a dualization scheme,
 to solve distributed optimization problems
in a network of dynamical nonlinear agents with a small duality gap. %They applied this technique
%to formation flight control.
In  \cite{yuan2011distributed,yi2015distributed,kia2015distributed},
to find the saddle point of the Lagrangian function,
a distributed gradient-based dynamics was developed for dual and primal
variables associated with each agent's constraint. In this approach, complexity
of the problem increases as the network grows in size and the number
of constraints increases. It is  worthwhile mentioning that, to deal with the consensus equality
constraint, the primal-dual approach yields linear terms associated
with this constraint. This restricts the obtained protocol from adopting nonlinear consensus strategies that
can in turn deliver fast convergence outcomes.
Besides, in the case of high order dynamics, this approach
does not work. To relax this restriction, one can
split the constrained distributed optimization problem into two parts, namely, a consensus subproblem
and local optimization ones,  see e.g \cite{rahili2015distributed}.
Then, the consensus subproblem can be dealt with
independently, and each agent's control law is obtained from the combination of the consensus protocol and other terms associated with
the local optimization problem. Following this line, the paper \cite{qiu2016distributed}
integrated a consensus protocol and a subgradient term into single-integrator
agents' control laws to tackle a distributed constrained optimal consensus
problem for single-integrator multi-agent systems with some common convex
set constraint. \cite{yang2016multi} exploited the same technique and
 presented a proportional-integral consensus protocol for
 distributed optimization problems with general constraints. Moreover, \cite{yang2016multi} relaxed the
assumption of global convexity on each local objective function to convexity on locally bounded feasible region.

Distributed optimal consensus  for double-integrator networks has been considered in few
papers, \textcolor{black}{see e.g} \cite{rahili2015distributed,xie2017global}. In \cite{rahili2015distributed},
a discontinuous nonlinear consensus protocol is combined with a distributed
gradient-based optimization algorithm to find the minimizer of a collective
smooth time-varying cost functions for two cases of single-integrator
networks and double-integrator networks. The authors of \cite{xie2017global}
proposed a bounded control law applied to a network of double-integrator agents, which
are supposed to reach consensus at a \textcolor{black}{value}
that minimizes the sum of local objective functions.
 In both above mentioned works, agents admit no constraint.

{\crb To the best knowledge of the authors, the optimal consensus problem with inequality constraints for networks with second-order agents has not been considered in details in the existing works. A solution to the optimal consensus problem for single- and double-integrator networks has already been developed by \cite{xie2017global}. However, these authors did not assume any constraint for the agents operating within the network. In practice, agents such as wheeled robots must admit constraints imposed by the field they move on. %Note that the problem of constrained consensus can be transformed into an optimal consensus problem subject to the constraints. % In such an optimization problem the local objective functions can be assumed to be linear or quadratic convex functions.\\
Furthermore, we develop a modified version of barrier method to solve the constrained optimal consensus problems.}
%In this problem each agent shall minimize the sum of private convex objective functions, with regard to private inequality constraints, and, simultaneously, reach a consensus on their positions through only sharing their information with their neighbors.

In this paper, we consider the constrained distributed optimal consensus
problem for  both single- and double-integrator networks, where each
agent is assigned with a convex objective function and \textcolor{black}{an inequality
constraint}. The main challenge to the double-integrator case is that one
does not have direct control on the positions of agents while the objective
function depends on the position of agents. \textcolor{black}{In this scenario, all}
 agents shall make a rendezvous at a point
that minimizes the sum of the individual uncoupled cost functions
and, simultaneously, satisfy all local inequality constraints.
%We assume two different scenarios of common and distinct constraints for each case.
To solve the present problem, we split it into \textcolor{black}{two subproblems, namely a consensus problem and individual
convex optimization ones.}
%To deal with the former problem, we utilize
%a continuous consensus protocol based on local information exchanged
%under an undirected graph. Then,
We exploit a slightly modified version of interior-point method to solve
the convex optimization subproblems. Moreover, to relax some of the restrictive  requirements imposed by this protocol,
%the assumption of Hessians of local objective and inequality constraint functions  being identical,
we  present a consensus-based distributed average tracking algorithm,
in which agents estimate components of the global objective function in a cooperative fashion.

This paper is structured as follows. The next section
reviews some background materials required in this paper.
We deal with the problem of distributed constrained optimal consensus for agents
with single-integrator dynamics in Section 3. In
Section 4, the same problem is investigated for the case of double-integrators.
 A numerical example is given in Section 5, and, finally, in Section 6, we present a conclusion for this paper.

\section{Notations and Preliminaries}\label{sec3}

In this section, we recall some preliminary lemmas
and concepts from graph theory, convex optimization, and stability
of dynamical systems which we will refer to later in this paper.

\subsection{Notations}

Throughout this paper, $\left\Vert \cdot\right\Vert _{p}$ and $\left\Vert \cdot\right\Vert $
denote p-norm and 2-norm operators, respectively. $\text{\ensuremath{\mathbb{{R}}}}$
represents the real numbers set and $\mathbb{{R}}^{+}$ implies the
positive real numbers subset. $\mathbb{{R}}^{N}$ includes all vectors
with $N$ real elements.
%\textcolor{black}{ $[\mathcal{{M}}_{ij}]_{N\times N}$ represents
%an $N\times N$ matrix with entries $\mathcal{{M}}_{ij}$, where the
%index $i$ stands for the $i$-th row and $j$ refers to $j$-th column.}
$\mathbb{{R}}^{N\times N}$ represents the set of all $N\times N$
matrices with real entries. $\left|\mathcal{{S}}\right|$ denotes
the cardinality of the set $\mathcal{{S}}$. For convenience, in the
sequel, set $sig(y)^{q}=|y|^q{sgn(y)}$ with $0<q<1$ and $y\in\mathbb{{R}}$.
$|y|$ is the absolute value of $y$ and $sgn(\cdot)$ is the
sign function. Note that for the vector valued arguments, $sig(\cdot)^{p}$ is
defined component-wise.

\subsection{Graph Theory }

Let $\mathcal{{G}=\left\{ \text{\ensuremath{\mathcal{V\text{,}E\text{,}A}}}\right\} }$
denote an undirected network, where $\mathcal{{V}}=\{\vartheta{}_{1},\cdots,\vartheta_{N}\}$
is the set of nodes and $\mathcal{{E}}\subseteq\mathcal{V\times V}$
represents the set of edges. \textcolor{black}{An edge (link) between
node $\vartheta_{i}$ and node $\vartheta_{j}$ is denoted by the
pair $(\vartheta_{i},\vartheta_{j})\in\mathcal{{E}}$, that indicates
that two nodes $\vartheta_{i}$ and $\vartheta_{j}$ exchange
information. \textcolor{black}{ Note that $(\vartheta_{i},\vartheta_{j})\in\mathcal{{E}}$ if and only if $(\vartheta_{j},\vartheta_{i})\in\mathcal{{E}}$.} The matrix $\mathcal{{A}=}[a_{ij}]_{N\times N}$ is the
adjacency matrix.} For an undirected graph, $\mathcal{{A}}$ is symmetric
and $a_{ij}=1$ means that $(\vartheta_{i},\vartheta_{j})\in\mathcal{{E}}$
and $a_{ij}=0$ indicates $(\vartheta_{i},\vartheta_{j})\notin\mathcal{{E}}$.
It is assumed that there is no self-loop, i.e. $a_{ii}=0$. The set
of neighbors of node $\vartheta_{i}$ is denoted by $\mathcal{{N}}_{\text{i}}=\{j\in\mathcal{\mathcal{{V}}}:(\vartheta_{i},\vartheta_{j})\in\mathcal{{E}}\}$.
{ Throughout this paper, we use the notation $\mathcal{N}$  to indicate the set $\{1,\cdots,N\}$, which is the set of all the indices assigned to all nodes.}
Assume an arbitrary orientation for the edges in $\mathcal{{G}}$,
then, $D=[d_{ik}]\in\mathbb{{R}}^{N\times|\mathcal{{E}}|}$ is the
incidence matrix associated with the undirected graph $\mathcal{{G}}$,
in which $d_{ik}=-1$ if the edge $(\vartheta_{i},\vartheta_{j})$
leaves node $\vartheta_{i}$, $d_{ik}=1$ if it enters the node, and
$d_{ik}=0$ otherwise. The Laplacian matrix $L=[l_{ij}]\in\mathcal{\mathbb{{R}}}^{N\times N}$
associated with the graph $\mathcal{{G}}$ is defined as $l_{ii}=\sum_{j=1,j\neq i}^{N}a_{ij}$
and $l_{ij}$$=-a_{ij}$ for $i\neq j$ . \textcolor{black}{Note that
$L=DD^{\top}$. If $\mathbf{\mathbf{1}\in\mathbb{R}^{N}}$ denotes
a vector of which all of entries are set to 1, then, $L\mathbf{1}=\mathbf{0}$
and $\mathbf{1}^{\top}L=\mathbf{0}$. }\textcolor{black}{All eigenvalues
of the Laplacian matrix $L$ are non-negative and }\textcolor{black}{it has }\textcolor{black}{only}\textcolor{black}{{} one zero eigenvalue
if the graph $\mathcal{{G}}$ }\textcolor{black}{is}\textcolor{black}{{}
connected. }We define consensus error in a network by $\bar{e}_{x}=\Pi\bar{x}$
where $\Pi=I_{N}-\frac{1}{N}\mathbf{1}_{N}\mathbf{1}_{N}^{\top}$,
and $\bar{x}$ denotes the aggregate state of the network as $\bar{x}=\left[x_{1}\ldots x_{N}\right]^{\top}$.
Note that $\mathbf{1}^{\top}\Pi=\mathbf{0}$ and $\Pi\mathbf{1=0}$.\textcolor{black}{{} }

The following lemma is crucial to some of the results studied
in this paper.
\begin{lemma}
 \label{lem:Courant-Fischer}(Courant-Fischer Formula) \cite{horn2012matrix}
Let A be an $n\times n$ real symmetric matrix with eigenvalues $\lambda_{1}\leq\lambda_{2}\leq\cdots\leq\lambda_{n}$
and corresponding eigenvectors $e_{1},\ldots,e_{n}$. Let $S_{k}$
denote the span of $e_{1},\ldots,e_{k}$ and $S_{k}^{\perp}$ denote
the orthogonal complement of $S_{k}$. Then, $\lambda_{k}=\underset{\underset{x\neq0}{x\in S_{k}^{\perp}}}{\text{min}}\frac{x^{T}Ax}{x^{T}x}$.
\end{lemma}

\subsection{Convex Optimization}

The differentiable function $\mathcal{{F}}(\cdot):\mathbb{R}^{n}\rightarrow\mathcal{\mathbb{R}}$
is convex if and only if $\mathcal{{F}}(w_2)\geq\mathcal{{F}}(w_1)+\nabla\mathcal{{F}}(w_1)^{\top}(w_2-w_1)$
for all $w_1,w_2\in\mathbb{R}^{n}$. The function $\mathcal{{F}}(\cdot):\mathbb{R}^{n}\rightarrow\mathcal{\mathbb{R}}$
is said strictly convex\textcolor{black}{{} if and only if $\mathcal{{F}}(w_2)>\mathcal{{F}}(w_1)+\nabla\mathcal{{F}}(w_1)^{\top}(w_2-w_1)$}
for all $w_1,w_2\in\mathbb{R}^{n}$. Consider the following convex optimization
problem with an inequality constraint

\begin{gather}
\text{{min}}\,\mathcal{{F}}({w}),\label{eq:Problem1}\\
\text{subject to }g_{i}(w)\leq0,\,\,\,i=1,\ldots,M,\nonumber
\end{gather}
where $\mathcal{{F}}(\cdot):\mathbb{R}^{n}\rightarrow\mathcal{\mathbb{R}}$
and $g_{i}(\cdot):\mathbb{R}^{n}\rightarrow\mathcal{\mathbb{R}}$
are both convex functions. The following lemma provides \textcolor{black}{the
condition for the optimal solution of }problem \eqref{eq:Problem1}.
\begin{lemma}
\cite[p. 243]{boyd2004convex} \label{lem:(KKT-Conditions)}(KKT Conditions)
Consider the convex optimization problem \eqref{eq:Problem1}. Assume
that  functions $\mathcal{{F}}\left(\cdot\right)$ and $g_{i}\left(\cdot\right)$
are continuously differentiable functions on $\mathbb{R}^{n}$ and there
exists $w^{*}\in\mathbb{R}^{n}$ such that $g_{i}(w^{*})\leq0$, $i=1,\ldots,M$.
$\mathcal{{F}}\left(\cdot\right)$ is also radially unbounded. Then,
$w^{*}$ is the optimal solution of the problem \eqref{eq:Problem1}
if and only if there exist some Lagrangian multipliers $\lambda_{i}^{*}>0$,
$i=1,\ldots,M,$ such that the following conditions are satisfied
\begin{gather}
g_{i}(w^{*})\leq0,\,\,\,\,\,\,\lambda_{i}^{*}g_{i}(w^{*})=0,\,i=1,\ldots,M,\label{eq:KKT1}\\
\nabla\mathcal{{F}}(w^{*})+\sum_{i=1}^{M}\lambda_{i}^{*}\nabla g_{i}(w^{*})=0.\label{eq:KKT2}
\end{gather}

\end{lemma}

\subsection{Stability of Dynamical Systems}

Consider the dynamical system
\begin{equation}
\dot{x}=f(x,t),\label{eq:Perl2}
\end{equation}
where $f(\cdot):\mathcal{D}\times[0,\infty)\rightarrow\mathbb{R}^{N}$
is piecewise continuous in $t$ and locally Lipschitz in $x$ on $\mathcal{D}\times[0,\infty)$,
and $\mathcal{D}\subset\mathbb{R}^{N}$ is a domain that contains the origin,
$x=0$.
\begin{lemma}
\label{lem:stability}\cite[Theorem 5.1]{khalil1996nonlinear} Let
$V:\mathcal{D}\times[0,\infty)\rightarrow\mathbb{R}^{N}$ be a continuously
differentiable function such that
\begin{eqnarray*}
W_{1}(x) & \leq & V(x,t)\leq W_{2}(x)\\
\frac{\partial V(x,t)}{\partial t}+\frac{\partial V}{\partial x}f(x,t) & \leq & -W_{3}(x),\,\,\,\forall\left\Vert x\right\Vert \geq\mu>0,
\end{eqnarray*}
$\forall t\geq0$, $\forall x\in \mathcal{D}$, where $W_{1}(x)$, $W_{2}(x)$,
and $W_{3}(x)$ are continuous positive definite functions on $\mathcal{D}$.
Take $r>0$ such that $B_{r}\subset \mathcal{D}$. Suppose that $\mu$ is small
enough such that
\[
\underset{\left\Vert x\right\Vert \leq\mu}{\text{{max}}}W_{2}(x)<\underset{\left\Vert x\right\Vert =r}{\text{min}}W_{1}(x)
\]
Let $\eta=\text{max}_{\left\Vert x\right\Vert \leq\mu}W_{2}(x)$
and take $\rho$ such that $\eta<\rho<\text{min}_{\left\Vert x\right\Vert =r}W_{1}(x)$.
Then, there exists a finite time $t_{1}$ (dependent on $x(t_{0})$
and $\mu$) such that $\forall\,x(t_{0})\in\{x\in B_{r}|W_{2}(x)\leq\rho\}$,
the solutions of $\dot{x}=f(x,t)$ satisfy $x(t)\in\{x\in B_{r}|W_{1}(x)\leq\rho\},\forall t\geq t_{1}$.
Moreover, if $\mathcal{D}=\mathbb{{R}}^{N}$ and $W_{1}(x)$ is radially unbounded,
then this result holds for any initial state and any $\mu$.
\end{lemma}

\section{Optimal Consensus for Single-integrator Dynamics}\label{sec4}

Consider $N$ dynamical agents under a network with the fixed topology
$\mathcal{{G}}$. Suppose that each agent is described by the continuous-time
single-integrator dynamics

\begin{equation}
\ensuremath{{\dot{x}_{i}}(t)={u_{i}}(t),}\quad\label{eq:1Agents_Dynamics}
\end{equation}
where $x_{i}(t)\in\mathbb{R}$ represents the position of agent $i$,
and $u_{i}(t)$ is the control input applied to agent $i$. In the rest of
this paper, notations $x_{i}$ and $x_{i}(t)$ are used interchangeably.
The same holds for $u_{i}$ and $u_{i}(t)$. Here, we consider only one dimensional agents for the sake of simplicity in notations.
However, it is straightforward to show that our algorithm can be extended to higher dimensional dynamics, i.e. the case where $x_i(t)\in\mathbb{R}^n$, as each dimension is decoupled from
others and, as a result, can be treated independently.  Each agent can share its state's information with agents within the
set of its neighbors, i.e. $\mathcal{{N}}_{i}$, based
on the graph $\mathcal{{G}}.$

The agents are supposed to rendezvous at a point, that
is the solution to the following convex optimization problem

\begin{equation}
\begin{array}{c}
\underset{x}{\text{min}}\,F(x)=\sum_{i=1}^{N}f_{i}(x),\\
\text{subject to}\,g_{i}(x)\leq0,\,i=1,\ldots,N,
\label{eq:Original_Prob}
\end{array}
\end{equation}
in which $f_{i}\left(\cdot\right):\mathbb{R\rightarrow\mathbb{R}}$
is the local objective function associated with node $\vartheta_{i}$
and $g_{i}\left(\cdot\right):\mathbb{R\rightarrow\mathbb{R}}$ represents
a constraint on the optimal position, associated with $i$-th agent.
{\crb Here, the variable $x$ is a scalar value that aims to minimize the global objective function in \eqref{eq:Original_Prob}.}
In other words, the agents shall meet each other in an optimum point
that fulfills all the constraint inequalities, i.e. $g_{i}(x)\leq0$,
$i\in\mathcal{N}$, and minimizes the aggregate objective function
$F(x)$. It is supposed that each agent only has knowledge of its
own local objective function as well as states information of those agents within the
set of its neighbors.

Note that solving the optimization problem \eqref{eq:Original_Prob}
in a centralized way requires knowledge of both the whole aggregate objective
function $\sum_{i=1}^{N}f_{i}(x)$ and all inequality constraints
$g_{i}(x)\leq0$, $i\in\mathcal{N}$.% and necessitates
%a fully connected graph.

{With considering the problem of consensus among the agents \eqref{eq:1Agents_Dynamics}, we reformulate the convex optimization problem} \eqref{eq:Original_Prob}
by
\begin{equation}
\begin{array}{c}
\underset{\underset{i=1,\cdots,N}{x_{i}}}{\text{min}}\sum_{i=1}^{N}f_{i}(x_{i}),\\
\text{subject to}\,\begin{cases}
g_{i}(x_{i})\leq0, & i=1,\ldots,N,\\
x_{i}=x_{j}.
\end{cases}\label{eq:Dis Min Prob}
\end{array}
\end{equation}
In the minimization problem \eqref{eq:Dis Min Prob}, the consensus
constraint, i.e. $x_{i}=x_{j},$ $\forall i,j\in\mathcal{N}$, is imposed
to guarantee that the same decision is made by all
agents eventually, and, subsequently, \textcolor{black}{ all agents rendezvous at the globally optimal point}.
In order to find the solution of the problem \eqref{eq:Dis Min Prob} \textcolor{black}{ in a distributed fashion, we  illustrate an algorithm in which}
each agent seeks the minimum of its own objective function, $f_{i}(x_{i})$,
fulfilling its associated inequality constraint, $g_{i}(x_{i})\leq0.$
Meanwhile, all agents exchange their states information through the
graph $\mathcal{{G}}$ to reach consensus on their position states. %In
%the sequel, we construct the control input $u_{i}$ such that the
%problem \eqref{eq:Dis Min Prob} \textcolor{black}{can be resolved in a distributed way.}

The following assumptions are considered in relation
to the optimization problem \eqref{eq:Dis Min Prob}.

\noindent \textit{Assumption 1.}
\renewcommand{\labelenumii}{\roman{enumii}}
\begin{enumerate}
\item[a.]  The objective functions, $f_{i}(\cdot)$, $i=1,\cdots,N$, are strictly
convex and twice continuously differentiable on $\mathbb{R^{\text{}}}$.
The functions $g_{i}(\cdot)$, $i=1,\cdots,N$, are convex and
twice continuously differentiable on $\mathbb{R^{\text{}}}$.
\item[b.] The global objective function $\sum_{i=1}^{N}f_{i}(x)$ is radially
unbounded, { with invertible Hessian $\sum_{i=1}^{N}\frac{\partial^{2}f_{i}({x})}{\partial{x^2}}$.}
\end{enumerate}
\textit{Assumption 2. (Slater's Condition)} There exists $x^{*}\in\mathbb{R}$
such that $g_{i}(x^{*})\leq0, \forall i\in\mathcal{N}$.

\noindent \textit{Assumption 3. }The graph $\mathcal{{G}}$ is undirected and
has one spanning tree.\\
%\begin{rmk}
%One may solve the constrained convex optimization problem \eqref{eq:Dis Min Prob}
%through a primal-dual approach by building an aggregate Lagrangian
%associated with the optimization problem. Then, some dynamics can be
%proposed to estimate values for local Lagrangian multipliers associated
%with the equality constraints, i.e. $x_{i}=x_{j},$ $i,j\text{\ensuremath{\in}}\{1,\ldots,N\}$,
%as well as other individual constraints from the original minimization
%problem, $g_{i}(x_{i})\leq0$, $i\in\{1,\ldots,N\}$.
%However, when networked agents are supposed to obey specific nonlinear
%consensus protocols in the control law, this approach does not work.
%That is why one needs to split the constrained distributed optimization
%problem \eqref{eq:Dis Min Prob} into a distributed optimization problem
%and a consensus one in order to deal with the consensus problem, separately.
%\end{rmk}

Intuitively, one can regard that the problem \eqref{eq:Dis Min Prob} consists of a convex
\textcolor{black}{optimization} problem, with inequality constraints,
and a consensus problem. The convex constrained optimization
problem can be defined as

\begin{equation}
\begin{array}{c}
\underset{\underset{i=1,\cdots,N}{x_{i}}}{\text{min}}\sum_{i=1}^{N}f_{i}(x_{i}),\\
\text{subject to}\,g_{i}(x_{i})\leq0,\,i=1,\ldots,N,
\end{array}\label{eq:Problem3}
\end{equation}
 while the consensus problem is
\begin{equation}
\underset{{\scriptstyle t\rightarrow\infty}}{lim}\,(x_{i}-x_{j})=0,\,\,\,\,i,j=1,\ldots,N.\label{eq:consensusProb}
\end{equation}
The convex optimization problem \eqref{eq:Problem3} can be reformulated
as follows,

\begin{alignat}{1}
\begin{array}{c}
\underset{\underset{i=1,\cdots,N}{x_{i}}}{\text{min}}\sum_{i=1}^{N}f_{i}(x_{i})-{\displaystyle \frac{\alpha}{\tau}}ln\left(-g_{i}(x_{i})\right).\end{array}\label{eq:1BarrierProb1}
\end{alignat}
 where $\tau\in\mathbb{{R}}^{+}$ and $\alpha>1$. The term $-ln\left(-g_{i}(x_{i})\right)$
is referred to as \textit{\textcolor{black}{logarithmic barrier }}\textcolor{black}{function}\textit{\textcolor{black}{.}}
Note that the domain of the logarithmic barrier function is the set of strictly
feasible points, i.e. $x_{i}\in\left\{ z\in\mathbb{{R}}:\,g_{i}(z)<0\right\} $.
The logarithmic barrier is a convex function; hence,
the new optimization problem is still a convex one.

Consider the objective function given in \eqref{eq:1BarrierProb1}.
It is easy to see that as $x_{i}$ approaches the hyperplane $g_{i}(x_{i})=0$,
the logarithmic barrier $-ln\left(-g_{i}(x_{i})\right)$ becomes extremely
large. Thus, it keeps the search domain within the strictly feasible
set. Note that the initial estimate shall be feasible, i.e. $g_{i}\left(x_{i}(0)\right)<0$,
$i\in\mathcal{N}$.

Suppose that the solutions to the optimization problem \eqref{eq:Problem3}
and \eqref{eq:1BarrierProb1} be $x^{*}$and $\widetilde{x}^{*}$,
respectively. Then, it can be shown that $f_{i}(x^{*})-f_{i}(\widetilde{x}^{*})=\frac{\alpha}{\tau}$
\cite{boyd2004convex,wang2011control}. This suggests a very straightforward
method for obtaining the solution to \eqref{eq:Problem3} with an
accuracy of $\varepsilon$ by choosing $\tau\geq\frac{\alpha}{\varepsilon}$
and solving \eqref{eq:1BarrierProb1}. Consequently, as $\tau$ increases,
the solution to the \textcolor{black}{optimization} problem \eqref{eq:1BarrierProb1}
becomes closer to the solution to \eqref{eq:Problem3}, i.e. as $\tau\rightarrow\infty$,
$f_{i}(x^{*})-f_{i}(\widetilde{x}^{*})\rightarrow0$ is concluded
\cite[pp. 568-571]{boyd2004convex}. In the literature,
this approach to solve inequality-constrained convex minimization
problems is known as \textit{interior-point} method. \textcolor{black}{}\footnote{\textcolor{black}{Interior-point method was first proposed by Fiacco et al.
in \cite{fiacco1990nonlinear} and is originally based on solving
a sequential unconstrained optimization problems, of which at every
sequence the value of $\tau$ increases. In this method, the last
point found in the previous step is used as the starting point for
the next one, and it goes until $\tau\geq\frac{\alpha}{\varepsilon}$. }}

We now express optimality conditions (so-called centrality conditions)
for the convex optimization problem \eqref{eq:1BarrierProb1} as

\begin{equation}
\begin{array}{c}
{\displaystyle \sum_{i=1}^{N}\frac{\partial f_{i}(\tilde{x}_{i}^{*})}{\partial x_{i}}}-{\displaystyle \frac{\alpha}{\tau}\frac{{\scriptstyle {\displaystyle {\displaystyle \frac{\partial g_{i}(\tilde{x}_{i}^{*})}{\partial x_{i}}}}}}{g_{i}(\tilde{x}_{i}^{*})}}=0,\\
g_{i}(\tilde{x}_{i}^{*})\leq0.
\end{array}\label{eq:1OptCondition}
\end{equation}
Given KKT conditions, i.e \eqref{eq:KKT1} and \eqref{eq:KKT2}, one
can define a dual variable as $\lambda_{i}=\frac{-{\textstyle \frac{\alpha}{\tau}}}{g_{i}(x_{i})}$,
then, according to \eqref{eq:KKT1}, it can be said that $\lambda_{i}^{*}g_{i}(\tilde{x}_{i}^{*})={\displaystyle \frac{-\alpha}{\text{\ensuremath{\tau}}}}$
and as $\tau\rightarrow\infty$, \eqref{eq:KKT1} is satisfied. Hence,
the solution to the problem \eqref{eq:1BarrierProb1} converges to
that of \eqref{eq:Problem3} as $\tau\rightarrow\infty$. Now, { we exploit an extended version of the interior point method} to
redefine the problem \eqref{eq:1BarrierProb1} as
\begin{equation}
\underset{\underset{i=1,\cdots,N}{x_{i}}}{\text{min}}\sum_{i=1}^{N}f_{i}(x_{i})-{\displaystyle \frac{\alpha}{t+1}}ln\left(-g_{i}(x_{i})\right)\label{eq:1BarrierProb2}
\end{equation}
Then, we propose the following control law to find the solution to the
optimization problem \eqref{eq:1BarrierProb2},

\begin{eqnarray}
u_{i} = -\left(\frac{\partial^{2}L_{i}(x_{i},t)}{\partial x_{i}^{2}}\right)^{-1}\left(\frac{\partial L_{i}(x_{i},t)}{\partial x_{i}}+\frac{\partial^{2}L_{i}(x_{i},t)}{\partial t\partial x_{i}}\right)+r_{i},\,\,i=1,\ldots,N,\label{eq:1Cont Law}
\end{eqnarray}
where
\begin{equation}
L_{i}(x_{i},t)=f_{i}(x_{i})-\frac{\alpha}{t+1}ln\left(-g_{i}(x_{i})\right),\label{eq:NewObjectiveFunc}
\end{equation}
 and
\begin{equation}
r_{i}=-\beta_{1}\sum_{j\in\mathcal{{N}}_{i}}tanh\beta_{2}(x_{i}-x_{j}),\label{eq:1consensus_prot}
\end{equation}
 in which $\beta_{1},\beta_{2}\in\mathbb{R}^{+}$.

Note that the control law \eqref{eq:1Cont Law} consists of two
parts: the first term is to minimize the local objective function,
and the second term is associated with the consensus
error.

First of all, we illustrate through the following lemma that the positions
of agents, i.e. $x_{i}$, $i\in\mathcal{N}$, reach a consensus
value under the control law \eqref{eq:1Cont Law}. In
the following, we introduce the notion of practical consensus. This
helps us to later show that all agents attain the same position perhaps
with arbitrarily small error.
\begin{definition}
{A network of agents with single-integrator dynamics
as \eqref{eq:1Agents_Dynamics} are said to achieve a }\textit{{practical
consensus if $\left|x_{i}(t)-x_{j}(t)\right|\leq\delta_{0}$}}, {$\forall i,j\in\mathcal{N}$
for an arbitrarily small $\delta_{0}$.}\end{definition}
\begin{lemma}
\label{prop:consensustheorem}Consider Assumptions 1.a and 3. \textcolor{black}{If
$\left|\omega_{i}-\omega_{j}\right|<\omega_{0}$, $\forall i,j\in\mathcal{N}$,
with $\omega_{i}=-\left(\frac{\partial^{2}L_{i}(x_{i},t)}{\partial x_{i}^{2}}\right)^{-1}\left(\frac{\partial L_{i}(x_{i},t)}{\partial x_{i}}+\frac{\partial^{2}L_{i}(x_{i},t)}{\partial t\partial x_{i}}\right)$,
and $\beta_{1}\sqrt{\lambda_{2}(L)}>\omega_{0}$, then, }there exist
some $t_{k}$ and $\delta_{0}>0$ such that the positions of the agents
with dynamics \eqref{eq:1Agents_Dynamics} {under
the control law \eqref{eq:1Cont Law} yield practical consensus,
i.e. $\left|x_{i}(t)-x_{j}(t)\right|\leq\delta_{0}$}{,
$\forall i,j\in\mathcal{N}$, }{for $t>t_{k}$}.\end{lemma}
\begin{proof}
The aggregate dynamics of agents \eqref{eq:1Agents_Dynamics} under
the control law \eqref{eq:1Cont Law} can be written as
\begin{equation}
\dot{\bar{x}}=-\beta_{1}Dtanh\left(\beta_{2}D^{\top}\bar{x}\right)+\Omega,\label{eq:1ConsLemma1}
\end{equation}
where $\varOmega=\left[\omega_{1}\ldots\omega_{N}\right]^{\top}$.
Let the network's consensus error be defined as $\bar{e}_{x}=\Pi\bar{x}$.
{Hence, one attains}

\begin{equation}
\dot{\bar{e}}_{x}=-\beta_{1}Dtanh\left(\beta_{2}D^{\top}\bar{e}_{x}\right)+\Pi\Omega.\label{eq:1ConsLemma2}
\end{equation}
Choose the Lyapanov candidate function
\begin{equation}
V(\bar{e}_{x})=\frac{1}{2}\bar{e}_{x}^{\top}\bar{e}_{x}.\label{eq:1ConsLemma3}
\end{equation}
By taking time derivative from $V(\bar{e}_{x})$ along the trajectories
of $\bar{e}_{x}$, it holds that

\begin{equation}
\dot{V}(\bar{e}_{x})=-\beta_{1}\bar{e}_{x}^{\top}D\,tanh\left(\beta_{2}D^{\top}\bar{e}_{x}\right)+\bar{e}_{x}^{\top}\Pi\Omega.
\end{equation}
Define $\bar{y}=D^{\top}\bar{e}_{x}$, where $\bar{y}=[y_{1}\ldots y_{N}]^{\top}$.
Then, {it is easy to see} that $\bar{y}^{\top}tanh(\beta_{2}\bar{y})=\sum_{i}y_{i}tanh(\beta_{2}y_{i})$.
From the inequality $-\eta tanh(\frac{\eta}{\epsilon})+\left|\eta\right|<0.2785\epsilon$
for some $\epsilon,\eta\in\mathbb{{R}}$ \cite{polycarpou1993robust},
it is straightforward to show that $-\bar{e}_{x}^{\top}D\,tanh\left(\beta_{2}D^{\top}\bar{e}_{x}\right)<-\left\Vert D^{\top}\bar{e}_{x}\right\Vert _{1}+\frac{N}{\beta_{2}}0.2785$.
Thus,

\begin{eqnarray}
\dot{V}(\bar{e}_{x}) & < & -\beta_{1}\left\Vert D^{\top}\bar{e}_{x}\right\Vert _{1}+\frac{\beta_{1}N}{\beta_{2}}0.2785+\left\Vert \bar{e}_{x}\right\Vert \left\Vert \Pi\Omega\right\Vert ,\\
 & \leq & -\beta_{1}\left\Vert D^{\top}\bar{e}_{x}\right\Vert +\frac{\beta_{1}N}{\beta_{2}}0.2785+\left\Vert \bar{e}_{x}\right\Vert \left\Vert \Pi\Omega\right\Vert ,
\end{eqnarray}
The second inequality arises from the inequality $\left\Vert \cdot\right\Vert \leq\left\Vert \cdot\right\Vert _{1}$. Then, from the assumption
$\left|\omega_{i}-\omega_{j}\right| <\omega_{0},\,\forall i,j\in\mathcal{N}$,
we conclude that

\begin{equation}
\dot{V}(\bar{e}_{x})<-\beta_{1}\sqrt{\bar{e}_{x}^{\top}DD^{\top}\bar{e}_{x}}+\frac{\beta_{1}N}{\beta_{2}}0.2785+\left\Vert \bar{e}_{x}\right\Vert \omega_{0}.
\end{equation}
According to Lemma \ref{lem:Courant-Fischer}, one can observe that
$\bar{e}_{x}^{\top}DD^{\top}\bar{e}_{x}\geq\lambda_{2}(L)\left\Vert \bar{e}_{x}\right\Vert ^{2}$, thus,

\[
\dot{V}(\bar{e}_{x})<-\beta_{1}\sqrt{\lambda_{2}(L)}\left\Vert \bar{e}_{x}\right\Vert +\frac{\beta_{1}N}{\beta_{2}}0.2785+\left\Vert \bar{e}_{x}\right\Vert \omega_{0}.
\]

From the statement of Lemma, we have $\beta_{1}\sqrt{\lambda_{2}(L)}>\omega_{0}$.
For $\left\Vert \bar{e}_{x}\right\Vert >\frac{\frac{\beta_{1}N}{\beta_{2}}0.2785}{\beta_{1}\sqrt{\lambda_{2}(L)}-\omega_{0}}$,
we obtain $\dot{V}(\bar{e}_{x})<0$. Now, we are ready to invoke
Lemma \ref{lem:stability}. It guarantees that by choosing $\beta_{2}$
large enough, one can make the consensus error $\delta_{0}$ as small
as desired. {Thus}, the proof is concluded.
\end{proof}

\begin{remark}
The assumption {$\left|\omega_{i}-\omega_{j}\right|<\omega_{0}$,}
$\forall i,j\in\mathcal{N}$,{ in Lemma \ref{prop:consensustheorem}
may seem unreasonable } {as} {it
implies boundedness of agents' positions, $x_{i}$, $i\in\mathcal{N}$.
In the following lemma, we {demonstrate}
that the agents' positions {indeed}
stay bounded. { It is worth mentioning that, by choosing a conservative bound on $\omega_{0}$
one can adjust the protocol's parameters to reach consensus with desired accuracy as we already showed in the proof of Lemma \ref{prop:consensustheorem}.}}
\end{remark}
\begin{lemma}
\label{lem:Boundedness}Consider the dynamics \eqref{eq:1Agents_Dynamics}
driven by the control law \eqref{eq:1Cont Law}. Then, under Assumptions
1.a and 3, the solutions of \eqref{eq:1Agents_Dynamics} are globally
bounded.\end{lemma}
\begin{proof}
We study boundedness of the solutions of dynamics \eqref{eq:1Agents_Dynamics}
under the control law \eqref{eq:1Cont Law} via the Lyapunov stability
analysis. Let us define a quadratic Lyapunov function as
\begin{equation}
W(\bar{x})=\frac{1}{2}(\bar{x}-\bar{x}^{*})^{\top}(\bar{x}-\bar{x}^{*}),\label{eq:LyapFunBounded}
\end{equation}
where $\bar{x}^{*}\in\mathbb{{R}}^{n}$ is the optimum point for the
convex function $\sum_{i=1}^{N}L_{i}(x_{i},t)$. Let us take derivative
from both sides of \eqref{eq:LyapFunBounded} along the trajectories
\eqref{eq:1Agents_Dynamics} %under the control law \eqref{eq:1Cont Law}
with respect to time. Then, we obtain
\begin{align}
\dot{W}(\bar{x}) & =(\bar{x}-\bar{x}^{*})^{\top}\dot{\bar{x}}\nonumber \\
 & =-\sum_{i=1}^{N}\left(x_{i}-x_{i}^{*}\right)\left(\frac{\partial L_{i}(x_{i},t)}{\partial x_{i}}+\frac{\partial^{2}L_{i}(x_{i},t)}{\partial t\partial x_{i}}\right)\left(\frac{\partial^{2}L_{i}(x_{i},t)}{\partial x_{i}^{2}}\right)^{-1}\nonumber \\
 & -\beta_{1}(\bar{x}-\bar{x}^{*})^{\top}Dtanh\beta_{2}D^{\top}\bar{x}\\
 & =-\sum_{i=1}^{N}\left(x_{i}-x_{i}^{*}\right)\left({\displaystyle \frac{\partial f_{i}(x_{i})}{\partial x_{i}}}-{\displaystyle \frac{t+\alpha-1}{\left(t+1\right)^{2}}\frac{{\scriptstyle {\displaystyle {\displaystyle \frac{\partial g_{i}(x_{i})}{\partial x_{i}}}}}}{g_{i}(x_{i})}}\right)\left(\frac{\partial^{2}L_{i}(x_{i},t)}{\partial x_{i}^{2}}\right)^{-1}\nonumber \\
 & -\beta_{1}(\bar{x}-\bar{x}^{*})^{\top}Dtanh\beta_{2}D^{\top}\bar{x}.\label{eq:1BoundedLemma2}
\end{align}
In order to obtain \eqref{eq:1BoundedLemma2},
we substituted $L_{i}(x_{i},t)$ from \eqref{eq:NewObjectiveFunc}.
Define $\hat{L}_{i}(x_{i},t)=f_{i}(x_{i})-\frac{t+\alpha-1}{\left(t+1\right)^{2}}ln\left(-g_{i}(x_{i})\right)$.
{Note that the function $\hat{L}_{i}(x_{i},t)$
is strictly convex in $x_{i}$ as }\textcolor{black}{$\alpha>1$.
Let the minimizer of $\hat{L}_{i}(x_{i},t)$ be $\hat{x}_{i}^{*}$.
One can deduce that} {the minimizers
of $\hat{L}_{i}(x_{i},t)$ and $L_{i}(x_{i},t)$
are the same, i.e. $\hat{x}_{i}^{*}=x_{i}^{*}$. On the other hand,
due to convexity of $\hat{L}_{i}(x_{i},t)$ in $x_{i}$, it holds
that $-\left(x_{i}-x_{i}^{*}\right)\frac{\partial\hat{L}_{i}(x_{i},t)}{\partial x_{i}}<\hat{L}_{i}(x_{i}^{*},t)-\hat{L}_{i}(x_{i},t)$,
$i=1,\ldots,N$. As the inequality $\hat{L}_{i}(x_{i}^{*},t)\leq\hat{L}_{i}(x_{i},t)$
holds for any $x_{i}$, it can be inferred that the first term on
the right side of the equality \eqref{eq:1BoundedLemma2} is non-positive.
Thus,} we obtain

\begin{eqnarray}
\dot{W}(\bar{x}) & \leq & -\beta_{1}(\bar{x}-\bar{x}^{*})^{\top}Dtanh\beta_{2}D^{\top}\bar{x}\nonumber \label{eq:1BoundedLemma}\\
 & = & -\beta_{1}\bar{x}^{\top}Dtanh\beta_{2}D^{\top}\bar{x}+\beta_{1}\bar{x}^{*\top}Dtanh\beta_{2}D^{\top}\bar{x}\\
 & \leq & -\beta_{1}\left\Vert D^{\top}\bar{x}\right\Vert _{1}+\frac{0.2785\beta_{1}N}{\beta_{2}}+\beta_{1}\left\Vert D^{\top}\bar{x}^{*}\right\Vert
\end{eqnarray}
The last \textcolor{black}{inequality} arises from the inequalities
$-\eta tanh(\frac{\eta}{\epsilon})+\left|\eta\right|<0.2785\epsilon$,
with $\epsilon, \eta\in\mathbb{{R}}$, which was introduced
in Lemma \ref{prop:consensustheorem}, and $\left\Vert tanh(\cdot)\right\Vert \leq1$. Furthermore,
one can say that $\left\Vert D^{\top}\bar{x}^{*}\right\Vert \leq d$,
with $d\in\mathbb{R}$. Therefore,
\begin{eqnarray}
\dot{W}(\bar{x}) & \leq & -\beta_{1}\left\Vert D^{\top}\bar{x}\right\Vert +\frac{0.2785\beta_{1}N}{\beta_{2}}+\beta_{1}d\label{eq:boundedEnd}\\
 & = & -\beta_{1}\sqrt{\bar{x}DD^{\top}\bar{x}}+\frac{0.2785\beta_{1}N}{\beta_{2}}+\beta_{1}d\\
 & = & -\theta_{1}\left\Vert \bar{x}\right\Vert +\left(\theta_{1}-\beta_{1}\sqrt{\lambda_{2}(DD^{\top})}\right)\left\Vert \bar{x}\right\Vert +\frac{0.2785\beta_{1}N}{\beta_{2}}+\beta_{1}d,\,\\
 & \leq & -\theta_{1}\left\Vert \bar{x}\right\Vert ,\,\,\,\forall\bar{x}\in\mathcal{{B}}.
\end{eqnarray}
where $\mathcal{{B}}=\left\{ \bar{x}\in\mathbb{{R}}^{N}|\left\Vert \bar{x}\right\Vert \geq\frac{\frac{0.2785\beta_{1}N}{\beta_{2}}+d\beta_{1}}{\beta_{1}\sqrt{\lambda_{2}(L)}-\theta_{1}}\right\} $ and $0<\theta_{1}<1$.
Now, by Lemma \ref{lem:stability}, it is certified
that $\bar{x}$ remains bounded.
\end{proof}

For the rest of this section, we found it convenient and illustrative to split our analysis
into two parts. We first study the case when all agents share
a common constraint, i.e. {$g_{i}(\cdot)=g_{j}(\cdot),\,\forall i,j\in\mathcal{N}$},
and then attend to the case when the agents have distinct constraints.

\subsection{Case I: Interconnected Agents with Common Constraints}

Here, {we assume that $g_i(\cdot) = g(\cdot), \forall i\in\mathcal{N}$, where $g:\mathbb{{R}}\rightarrow\mathbb{R}$,
represents a common twice differentiable convex inequality constraint} associated with all agents, and present a theorem which asserts that
the control law \eqref{eq:1Cont Law} drives all the agents to the
optimal solution of the \textcolor{black}{optimization} problem \textcolor{black}{\eqref{eq:1BarrierProb2}. }
\begin{theorem}
\label{thm:Theorem1} Assume that Assumptions 1, 2, and 3 hold. Then,
under the control law \eqref{eq:1Cont Law}, \textcolor{black}{agents
with dynamics} \eqref{eq:1Agents_Dynamics} will converge to a point
that is the solution to the optimization problem \textcolor{black}{\eqref{eq:1BarrierProb2}}
if $\frac{\partial^{2}f_{i}(x_{i})}{\partial x_{i}^{2}}=\frac{\partial^{2}f_{j}(x_{j})}{\partial x_{j}^{2}}$
$\forall i,j\in\mathcal{N} $. \end{theorem}
\begin{proof}
Define the candidate time-varying Lyapunov function as
\begin{equation}
V(\bar{x},t)=\frac{1}{2}\left(\sum_{i=1}^{N}\frac{\partial L_{i}(x_{i},t)}{\partial x_{i}}\right)^{2}.\label{eq:1Theorem1}
\end{equation}
By taking derivative from $V(\bar{x},t)$ with respect to time along
with the trajectories described by \eqref{eq:1Agents_Dynamics} and
\eqref{eq:1Cont Law}, it holds that
\begin{eqnarray*}
\dot{V}(\bar{x},t) & = & \sum_{i=1}^{N}\frac{\partial L_{i}(x_{i},t)}{\partial x_{i}}\left(\sum_{i=1}^{N}\frac{\partial^{2}L_{i}(x_{i},t)}{\partial x_{i}^{2}}u_{i}+\frac{\partial^{2}L_{i}(x_{i},t)}{\partial t\partial x_{i}}\right)
\end{eqnarray*}
By substituting $u_{i}$ in the above equation \textcolor{black}{from}
\eqref{eq:1Cont Law}, we have
\begin{eqnarray}
\dot{V}(\bar{x},t) & =- & \left(\sum_{i=1}^{N}\frac{\partial L_{i}(x_{i},t)}{\partial x_{i}}\right)^{2}\label{eq:1Theorem2}\\
 & \leq & 0\label{eq:1Theorem3}
\end{eqnarray}
The equation \eqref{eq:1Theorem2} is certified by the assumption
$\frac{\partial^{2}f_{i}(x_{i})}{\partial x_{i}^{2}}=\frac{\partial^{2}f_{j}(x_{j})}{\partial x_{j}^{2}}$,
{$\forall i,j\in\mathcal{N}$}, (that results
in $\frac{\partial^{2}L_{i}(x_{i},t)}{\partial x_{i}^{2}}=\frac{\partial^{2}L_{j}(x_{j},t)}{\partial x_{j}^{2}},$$\forall i,j\in\mathcal{N}$)
and the fact that $\sum_{i=1}^{N}r_{i}=0$. From the inequality \eqref{eq:1Theorem3},
it holds that $\sum_{i=1}^{N}\frac{\partial L_{i}(x_{i},t)}{\partial x_{i}}$
remains bounded in $\mathbb{{R}}^{N}\bigcup\{\infty\}$, i.e. it belongs
to $\mathcal{{L}^{\infty}}$ space. { One can integrate both sides of
equality \eqref{eq:1Theorem2} with respect to time. Then, according to the inequality \eqref{eq:1Theorem3},
the following must hold}
\begin{equation}
\int_{0}^{R}\left(\sum_{i=1}^{N}\frac{\partial L_{i}(x_{i},t)}{\partial x_{i}}\right)^{2}dt=-\int_{0}^{R}\dot{V}(\bar{x},t)dt=-V(\bar{x}(R),R)+V(\bar{x}(0),0)\leq V(\bar{x}(0),0).\label{eq:2Thorem4}
\end{equation}
Hence, $\sum_{i=1}^{N}\frac{\partial L_{i}(x_{i},t)}{\partial x_{i}}\in\mathcal{{L}}^{2}$.
We now invoke Barbalat's Lemma \cite{tao1997simple}
and claims that $\sum_{i=1}^{N}\frac{\partial L_{i}(x_{i},t)}{\partial x_{i}}$
asymptotically converges to zero as $t\rightarrow\infty$. Thereby,
the first optimality condition in \eqref{eq:1OptCondition} is asymptotically
satisfied.

In the remainder the proof, we show that the second optimality condition
in \eqref{eq:1OptCondition} also
holds. Suppose that $g\left(x_{i}(0)\right)<0$, $\forall i\in\mathcal{N}$.
We will do the proof by contradiction to illustrate that $g\left(x_{i}(t)\right)<0$
for $t>0$. Assume that we had $g\left(x_{i}(t_{0}^{-})\right)<0$
and $g\left(x_{i}(t_{0}^{+})\right)>0$ for some $i$ and a finite
$t_{0}>0$. Due to continuity of the function $g\left(\cdot\right)$,
$g(x_{i}(t_{0}))$ would be zero. This implies that $\sum_{i=1}^{N}\frac{\partial L_{i}(x_{i},t)}{\partial x_{i}}$
becomes unbounded at $t_{0}$ that contradicts the fact that $\sum_{i=1}^{N}\frac{\partial L_{i}(x_{i},t)}{\partial x_{i}}\in\mathcal{{L}^{\infty}}$, { achieved earlier.
Hence, the inequality $g(x_{i}(t))<0$ with $g\left(x_{i}(0)\right)<0$
holds for $t>0$. Thereby, the proof is established.}

\end{proof}

One should note that through Lemma \ref{prop:consensustheorem}, we
showed practical consensus on states.
Furthermore, by Theorem \ref{thm:Theorem1}, we proved that the control
laws \eqref{eq:1Cont Law} solve the optimization problem \eqref{eq:Problem3}
on the conditions {that} $g_{i}(x_{i})=g_{j}(x_{j})$
{and} $\frac{\partial^{2}f_{i}(x_{i})}{\partial x_{i}^{2}}=\frac{\partial^{2}f_{j}(x_{j})}{\partial x_{j}^{2}},\,\forall i,j\in\mathcal{N}$.
The {condition} $\frac{\partial^{2}f_{i}(x_{i})}{\partial x_{i}^{2}}=\frac{\partial^{2}f_{j}(x_{j})}{\partial x_{j}^{2}}$,
$\forall i,j\in\mathcal{N}$, {may} first seem
strong; {however, it is} feasible in many problems,
e.g. the convex functions that belong to the set $\{f_{i}(\cdot)\mid f_{i}(x_{i})=(x_{i}-a_{i})^{2},a_{i}\in\mathbb{{R}}\}$
meet this requirement. To relax this assumption
and the condition of local constraints being the same, in
the following subsection, we will present an estimation-based approach
to solve the distributed optimization problem \eqref{eq:Dis Min Prob}.
This algorithm was initially proposed in \cite{rahili2015distributed}
and we adopt it here to relax the constraint $\frac{\partial^{2}L_{i}(x_{i},t)}{\partial x_{i}^{2}}=\frac{\partial^{2}L_{j}(x_{j},t)}{\partial x_{j}^{2}},\,\forall i,j\in\mathcal{N},$
which { may not be fulfilled in some cases.}

\subsection{Case II: Agetns with Distinct Constraints}

In the sequel, we first propose a centralized paradigm
to find the solution of a typical optimization problem associated
with a network under the graph $\mathcal{{G}}$.
Next, we adopt the technique of distributed average
tracking to estimate the parameters of the proposed centralized
control law in a { cooperative manner}. This approach drives {all}
the agents towards the solution {of} the optimization
problem \eqref{eq:1BarrierProb2} and also yields consensus.

Consider the single-integrator dynamics
\begin{equation}
\ensuremath{\begin{array}{l}
{\dot{x}}(t)={u}(t),\end{array}}\label{eq:1SingIntDym}
\end{equation}
where $u(t)\in\mathbb{{R}}$ and $x(t)\in\mathbb{{R}}$
denote the state and the control input, respectively. Consider an objective
function, say $Q(x,t):\mathbb{{R}\times\mathbb{{R}\rightarrow\mathbb{{R}}}}$,
that is twice continuously differentiable and strictly convex in $x$. { Moreover, it has one minimizer when $t\rightarrow\infty$, and its Hessian is invertible, i.e.
$\left(\frac{\partial^{2}Q(x,t)}{\partial x^{2}}\right)^{-1}, \forall{x},t,$ exists.}
In the following, we show that
\begin{equation}
u(t)=-\left(\frac{\partial^{2}Q(x,t)}{\partial x^{2}}\right)^{-1}\left(\frac{\partial Q(x,t)}{\partial x}+\frac{\partial^{2}Q(x,t)}{\partial x\partial t}\right)\label{eq:1CentLaw}
\end{equation}
 will make the dynamics \eqref{eq:1SingIntDym} converge to the minimizer
of the time-varying objective function $Q(x,t)$. Consider
the following Lyapunov function
\[
V(x,t)=\frac{1}{2}\left(\frac{\partial Q(x,t)}{\partial x}\right)^{2}
\]
and take its time derivative along the {trajectories}
of dynamics \eqref{eq:1SingIntDym}. Then, we have
\[
\dot{V}(x,t)=-\left(\frac{\partial Q(x,t)}{\partial x}\right)\left(\frac{\partial^{2}Q(x,t)}{\partial x^{2}}\dot{x}+\frac{\partial^{2}Q(x,t)}{\partial x\partial t}\right).
\]
 By substituting $\dot{x}$ in the above relation with \eqref{eq:1CentLaw},
the following is obtained
\[
\dot{V}(x,t)=-\left(\frac{\partial Q(x,t)}{\partial x}\right)^{2}\leq0.
\]
{Following the same reasoning as} in the proof of
Theorem \ref{thm:Theorem1}, it holds that ${\displaystyle \frac{\partial Q(x,t)}{\partial x}}\in\mathcal{{L}^{1}},\mathcal{{L}}^{2}$.
Then, by means of Barbalat's lemma, we have ${\displaystyle \frac{\partial Q(x,t)}{\partial x}}\rightarrow0$
as $t\rightarrow\infty$. {Thereby, the optimality
condition is} satisfied, i.e.
${\displaystyle \frac{\partial Q(x^{*},t)}{\partial x}=0}$.

Now, let us investigate a network of dynamical agents
with dynamics \eqref{eq:1Agents_Dynamics} {under
the topology} $\mathcal{{G}}$ {with} the collective
convex objective function $Q(x,t)=\sum_{i=1}^{N}L_{i}(x_{i},t)$.
From {the control law }\eqref{eq:1CentLaw}, one
{can readily} conclude that the control law
\begin{equation}
u_{i}(t)=-\left(\frac{\partial^{2}\sum_{i=1}^{N}L_{i}(x_{i},t)}{\partial x_{i}^{2}}\right)^{-1}\left(\frac{\partial\sum_{i=1}^{N}L_{i}(x_{i},t)}{\partial x_{i}}+\frac{\partial^{2}\sum_{i=1}^{N}L_{i}(x_{i},t)}{\partial x_{i}\partial t}\right)\label{eq:1CentLaw_Collective}
\end{equation}
{yields} the solution to the collective convex objective
function if Assumptions 1 and 2 hold. It is {apparent}
that the control law \eqref{eq:1CentLaw_Collective} is not locally
implementable since it requires the knowledge of the whole network
such as aggregate objective function $\sum_{i=1}^{N}L_{i}(x_{i},t)$. {With
the following algorithm, we provide an algorithm
that enables us to implement \eqref{eq:1CentLaw_Collective} in a
distributed manner such that the optimization problem }\eqref{eq:Dis Min Prob} is resolved.

As it follows, each agent generates an internal dynamics to obtain
the estimates of collective objective function's gradients and some other
terms, which are required for computation of \eqref{eq:1CentLaw_Collective}
{using} only local information. {Consider}
the following internal dynamics,

\begin{align}
\dot{\kappa}_{i}(t) & =-c\sum_{j\in\mathcal{{N}}_{i}}sgn\left(\nu{}_{i}(t)-\nu_{j}(t)\right),\label{eq:1EstDynamics1}
\end{align}
where
\begin{equation}
\nu_{i}(t)=\kappa_{i}(t)+\left[{\displaystyle \begin{array}{c}
\frac{\partial L_{i}(x_{i},t)}{\partial x_{i}}\\
\frac{\partial^{2}L_{i}(x_{i},t)}{\partial x_{i}\partial t}\\
\frac{\partial^{2}L_{i}(x_{i},t)}{\partial x_{i}^{2}}
\end{array}}\right].\label{eq:1EstDynamics2}
\end{equation}

From \eqref{eq:1EstDynamics1}, one obtains $\sum_{i=1}^{N}\dot{\kappa}_{i}(t)=0$.
Assume that $\kappa_{i}$, $\forall{i}\in\mathcal{N}$, are initialized such that
$\sum_{i=1}^{N}\kappa(0)=0$. Then, $\sum_{i=1}^{N}\kappa_{i}(t)=0$
is concluded for $t>0$. Hence, $\sum_{i=1}^{N}\nu_{i}(t)=\sum_{i=1}^{N}\left[{\displaystyle \begin{array}{c}
\frac{\partial L_{i}(x_{i},t)}{\partial x_{i}}\\
\frac{\partial^{2}L_{i}(x_{i},t)}{\partial x_{i}\partial t}\\
\frac{\partial^{2}L_{i}(x_{i},t)}{\partial x_{i}^{2}}
\end{array}}\right]$. It follows from Theorem 1 in \cite{chen2012distributed} that if
$c>\underset{t}{sup}\left\{ \left\Vert \kappa{}_{i}(x_{i},t)\right\Vert _{\infty}\right\} $,
$\forall i\in\mathcal{N}$, then consensus on $\text{\ensuremath{\nu}}_{i}(t)$,
$\forall{i}\in\mathcal{N}$, i.e. $\left|\nu_{i}(t)-\nu_{j}(t)\right|=0,\forall i,j\in\mathcal{N}$,
is achieved \textcolor{black}{over a finite time, say $T_{1}$.} With
$\nu_{i}(t)=\nu_{j}(t)$, {$\forall i,j\in\mathcal{N}$,}
the following holds,
\begin{equation}
\nu_{i}(t)=\frac{1}{N}\sum_{i=1}^{N}\left[{\displaystyle \begin{array}{c}
\nu_{i1}\\
\nu_{i2}\\
\nu_{i3}
\end{array}}\right],\label{eq:1EstDyn4}
\end{equation}
{where $\nu_{i1}=\frac{\partial L_{i}(x_{i},t)}{\partial x_{i}}$,
$\nu_{i2}=\frac{\partial^{2}L_{i}(x_{i},t)}{\partial x_{i}\partial t}$,
and $\nu_{i3}=\frac{\partial^{2}L_{i}(x_{i},t)}{\partial x_{i}^{2}}$. }

{We assert that the protocol }

\begin{equation}
u_{i}=-\nu_{i3}^{-1}\left(\nu_{i1}+\nu_{i2}\right)+r_{i},\,i=1,\ldots,N, \label{eq:1DisEstLaw}
\end{equation}
with $r_i$ as in \eqref{eq:1consensus_prot}{ will drive the agents with dynamics \eqref{eq:1Agents_Dynamics}
to the solution of the distributed convex optimization problem \eqref{eq:Dis Min Prob}.
Here, we omit the consensus analysis as it is would be identical to
the proof of Lemma \ref{prop:consensustheorem} with similar conditions.
We only present a lemma that shows how the protocol \eqref{eq:1DisEstLaw}
yields the solution to the optimization problem \eqref{eq:1BarrierProb2}. }
\begin{lemma}
Suppose that Assumptions 1, 2,and 3 hold, $\sum_{i=1}^{N}\kappa_{i}(0)=0$,
and $c>\underset{t}{sup}\left\{ \left\Vert \kappa{}_{i}(x_{i},t)\right\Vert _{\infty}\right\} ,\,\forall i\in\mathcal{N}$.
Then, the protocol \eqref{eq:1DisEstLaw} will solve the convex optimization
problem \eqref{eq:1BarrierProb2}. \end{lemma}
\begin{proof}
Let us define the following Lyapunov candidate function,
\begin{equation}
V(t)=\frac{1}{2}\left(\sum_{i=1}^{N}\nu_{i1}(t)\right)^{2}.
\end{equation}
After calculating time derivative of $V(t)$, the following holds,

\begin{align}
\dot{V}(t) & =\left(\sum_{i=1}^{N}\nu_{i1}(t)\right)\left(\sum_{i=1}^{N}\nu_{i3}(t)u_{i}(t)+\nu_{i2}(t)\right).
\end{align}
Form the control law \eqref{eq:1DisEstLaw}, we attain
\begin{eqnarray}
\dot{V}(t) & = & -\left(\sum_{i=1}^{N}\nu_{i1}(t)\right)^{2},
\end{eqnarray}
in which we used  the equalities $\nu_{i3}(t)=\nu_{j3}(t),\,\forall i,j\in\mathcal{N}$
for $t>T_{1}$, and $\sum_{i=1}^{N}r_{i}(t) =0$ for the graph $\mathcal{{G}}$.
We conclude that
\begin{eqnarray}
\dot{V}(t) & \leq & 0,\,\forall t>T_{1},\label{eq:1EstLyapfin}
\end{eqnarray}

On the other hand, we assert that $x_{i}$, $i\in\mathcal{N}$, stay
bounded after a finite time as the agents' dynamics is locally Lipschitz
and their inputs are bounded. {This} means that for
$t\leq T_{1}$, $x_{i}$ {remains finite, i.e. $x_{i}\in\mathbb{{R}},\forall i\in\mathcal{N}$}.
Hence, we can do stability analysis from $T_{1}$ onwards. {We
now appeal to the same justification as presented in the proof of
Theorem \ref{thm:Theorem1} and invoke Barbalat's lemma \cite{tao1997simple} to show}
that $\sum_{i=1}^{N}\nu_{i1}(t)=0$ as $t\rightarrow\infty$.
{The remainder of the proof is similar to that
of Theorem \ref{thm:Theorem1}.}

\end{proof}

\section{{Optimal Consensus for Double-integrator
Agents }}

This section investigates distributed optimal consensus
problem in a network of agents with double-integrator dynamics. The
final positions of agents shall be the minimizer of the network's {global}
objective function and satisfy some local constraints. However, {in
this case} we only have direct control over the velocity of each agent.
{This} makes the problem more challenging compared
{to that of the previous section. }

Consider a network of $N$ agents with double-integrator dynamics
as
\begin{equation}
\ensuremath{\begin{array}{l}
{{\dot{x}}_{i}}(t)={v_{i}}(t),\\
{{\dot{v}}_{i}}(t)={u_{i}}(t),
\end{array}}\label{eq:2AgentsDyn}
\end{equation}
where $x_{i}(t),v_{i}(t)\in\mathbb{R}$ are the position and velocity
of $i$-th agent, respectively. Moreover, $u_{i}(t)\in\mathbb{R}$
represents the control law. These agents exchange their positions'
information under the graph $\mathcal{{G}}$. The
goal is to design $u_{i}(t)$, $i\in\mathcal{N}$, in order to
find {the solution to the optimization problem} \eqref{eq:Dis Min Prob}.
To this end, we follow the same {strategy} that was
used in the previous section, {namely splitting the
problem }\eqref{eq:Dis Min Prob} into two subproblems, i.e. the convex
optimization problem \eqref{eq:Problem3}, and the following consensus
problem
\begin{equation}
\underset{t\rightarrow\infty}{lim}(x_{i}-x_{j})=0,\,\,\,\,\,\,\,\underset{t\rightarrow\infty}{lim}v_{i}=0,\,\,\,\,\forall i,j\in\{1,\ldots,N\}.\label{eq:2ConsProb}
\end{equation}
The problem in \eqref{eq:2ConsProb}
is {referred to as} stationary consensus problem
in the literature \cite{rezaee2015average}. {As
we have already shown}, the problem \eqref{eq:Problem3} can be redefined
{as the optimization problem} \eqref{eq:1BarrierProb2}
{via the interior-point method.}

In the sequel, we illustrate that if we choose the control input of
agent $i$ as
\begin{equation}
u_{i}(t)=-\frac{dk_{i}(x_{i},t)}{dt}-\frac{\partial^{2}L_{i}(x_{i},t)}{\partial x_{i}^{2}}\frac{\partial L_{i}(x_{i},t)}{\partial x_{i}}+r_{i},\label{eq:2InputLaw}
\end{equation}
 where
\begin{equation}
k_{i}(x_{i},t)=\left(\frac{\partial^{2}L_{i}(x_{i},t)}{\partial x_{i}^{2}}\right)^{-1}\left(\frac{\partial L_{i}(x_{i},t)}{\partial x_{i}}+\frac{\partial^{2}L_{i}(x_{i},t)}{\partial t\partial x_{i}}\right)\label{eq:k_i}
\end{equation}
and
\begin{gather}
r_{i}=-\gamma_{1}\sum_{j\in\mathcal{{N}}_{i}}sig(x_{i}-x_{j})^{q}-\gamma_{2}sig(v_{i})^{p},\label{eq:2ConsProt}
\end{gather}
with $0<q<1$, $p=\frac{2q}{q+1}$, and $\gamma_{1},\gamma_{2}\in\mathbb{R}^{+}$,
the trajectories of the dynamics {stated in} \eqref{eq:2AgentsDyn}
converge to the solution of the convex optimization problem \eqref{eq:1BarrierProb2}.
\textcolor{black}{Moreover, all agents attain the same position perhaps with arbitrarily small error and
asymptotically zero velocity.}
We first introduce the notion of practical stationary consensus  \textcolor{black}{to formalize the latter}.
\begin{definition}
{A network of agents with double-integrator dynamics
as in \eqref{eq:2AgentsDyn} are said to achieve a practical stationary
consensus if $|x_{i}(t)-x_{j}(t)|<\delta_{0}$, $\forall i,j\in\mathcal{N}$
for an arbitrarily small $\delta_{0}$ and $|v_{i}(t)|\le\delta_{1}, \forall i\in\mathcal{N}$,
for a small desired $\delta_{1}$. }\end{definition}
\begin{lemma}
\label{prop:ConsensusProp2}Consider Assumptions 1.a and 3. Suppose
that the agents \eqref{eq:2AgentsDyn} exchange their positions' information
according to the graph $\mathcal{{G}}$ under the protocol \eqref{eq:2InputLaw}.
\textcolor{black}{If }$\left|\varphi_{i}-\varphi_{j}\right|<\varphi_{0}$\textcolor{black}{,
$i,j\in\mathcal{N},$ where $\varphi_{i}=-\frac{dk_{i}(x_{i},t)}{dt}-\frac{\partial^{2}L_{i}(x_{i},t)}{\partial x_{i}^{2}}\frac{\partial L_{i}(x_{i},t)}{\partial x_{i}}$,
and }$\gamma_{2},\gamma_{1}\gg\varphi_{0}$, then, {
practical stationary consensus is achieved in a finite time. }\end{lemma}
\begin{proof}
Define the aggregate states by $\bar{x}=[x_{1}\cdots x_{N}]^{\top}\in\mathbb{R}^{N}$
and $\bar{v}=[v_{1}\cdots v_{N}]^{\top}\in\mathbb{R}^{N}$. We first
introduce the    following dynamics,
\begin{align}
\dot{\bar{x}} & =\bar{v},\nonumber \\
\dot{\bar{v}} & =-\gamma_{1}Dsig(D^{\top}\bar{x})^{q}-\gamma_{2}sig(\bar{v})^{p}.\label{eq:2ConsTh0}
\end{align}
{It was shown in \cite{wang2008finite} that the
above dynamics reaches the stationary consensus in some finite time, say $\tau_0$.
We now proceed with the rest of the proof by defining the error vector
associated with the position as }$\bar{e}_{x}=\bar{x}-\frac{1}{N}\mathbf{1}_{N}\mathbf{1}_{N}^{\top}\bar{x}$.
Then, by taking derivative from $\bar{e}_{x}$ with respect to time,
one obtains that $\dot{\bar{e}}_{x}=\bar{v}-\frac{1}{N}\mathbf{1}_{N}\mathbf{1}_{N}^{\top}\bar{v}$.
Let $\bar{e}_{v}=\dot{\bar{e}}_{x}$. Then, from \eqref{eq:2ConsTh0},
the aggregate consensus error dynamics is derived,
\begin{align}
\dot{\bar{e}}_{x} & =\bar{e}_{v},\label{eq:2ConsTh1}\\
\dot{\bar{e}}_{v} & =-\gamma_{1}Dsig(D^{\top}\bar{e}_{x})^{q}-\gamma_{2}\Pi sig(\bar{v})^{p}.\nonumber
\end{align}

{We now attain the aggregate model associated with
\eqref{eq:2AgentsDyn} and \eqref{eq:2InputLaw}. This model can be
seen as a perturbed form of the hypothetical nominal system in \eqref{eq:2ConsTh1},}
\begin{align}
\dot{\bar{e}}_{x} & =\bar{e}_{v},\label{eq:2ConsTh2}\\
\dot{\bar{e}}_{v} & =-\gamma_{1}Dsig(D^{\top}\bar{e}_{x})^{q}-\gamma_{2}\Pi sig(\bar{v})^{p}+\Pi\varPsi,\nonumber
\end{align}
where $\varPsi=[\varphi_{1}\ldots\varphi_{N}]^{\top}$ and $\Pi\varPsi\in\mathbb{R}^{N}$
refers to the perturbation term to the nominal system \eqref{eq:2ConsTh1}.
We choose the following Lyapunov candidate function as
\begin{equation}
V(\bar{e}_{x},\bar{e}_{v})=\gamma_{1}\sum_{i=1}^{N}\sum_{j\in\mathcal{N}_{i}}\int_{0}^{e{}_{x_{i}}-e_{x_{j}}}sig(s)^{q}ds+\frac{1}{2}\bar{e}_{v}^{\top}\bar{e}_{v}.\label{eq:2ConsTh3}
\end{equation}
One can take a time derivative of $V(\bar{e}_{x},\bar{e}_{v})$ and
obtain

\begin{equation}
\dot{V}(\bar{e}_{x},\bar{e}_{v})=\gamma_{1}\sum_{i=1}^{N}\sum_{j\in\mathcal{N}_{i}}sig(e_{x_{i}}-e_{x_{j}})^{q}e_{v_{i}}+\bar{e}_{v}^{\top}\dot{\bar{e}}_{v}.\label{eq:2ConsTh4}
\end{equation}
From \eqref{eq:2ConsTh2}, it holds that
\begin{eqnarray}
\dot{V}(\bar{e}_{x},\bar{e}_{v}) & = & \gamma_{1}\sum_{i=1}^{N}\sum_{j\in\mathcal{N}_{i}}sig(e_{x_{i}}-e_{x_{j}})^{q}e_{v_{i}}+\bar{e}_{v}^{\top}\left(-\gamma_{1}Dsig(D^{\top}\bar{e}_{x})^{q}-\gamma_{2}\Pi sig(\bar{v})^{p}+\Pi\varPsi\right)\nonumber \\
 & = &\gamma_{1} \sum_{i=1}^{N}e_{v_{i}}\sum_{j\in\mathcal{N}_{i}}sig(e_{x_{i}}-e_{x_{j}})^{q}-\gamma_{1}\bar{e}_{v}^{\top}Dsig(D^{\top}\bar{e}_{x})^{q}-\gamma_{2}\bar{v}^{\top}\Pi^{\top}\Pi sig(\bar{v})^{p}\nonumber \\
 &  & + \, \bar{v}^{\top}\Pi^{\top}\Pi\varPsi\nonumber \\
 & = & -\gamma_{2}\bar{v}^{\top}\Pi sig(\bar{v})^{p}+\bar{v}^{\top}\Pi\varPsi
\end{eqnarray}
Since the inequality $\bar{v}^{\top}\Pi sig(\bar{v})^{p}\geq{\textstyle \frac{N}{N-1}}\left(\left\Vert \bar{v}\right\Vert _{p+1}\right)^{p+1}$
always holds {[}Appendix{]}, the following inequality is concluded

\begin{eqnarray}
\dot{V} & \leq & -{\textstyle \frac{\gamma_{2}N}{N-1}}\left(\left\Vert \bar{v}\right\Vert _{p+1}\right)^{p+1}+\bar{v}^{\top}\Pi\varPsi\label{eq:2ConsTh5}\\
 & \leq & -{\textstyle \frac{\gamma_{2}N}{N-1}}\left(\left\Vert \bar{v}\right\Vert _{p+1}\right)^{p+1}+\varphi_{0}\left\Vert \bar{v}\right\Vert \label{eq:2ConsTh5.5} \\
 & \leq & -{\textstyle \frac{\gamma_{2}N}{N-1}}\left(\left\Vert \bar{v}\right\Vert _{p+1}\right)^{p+1}+\varphi_{0}\left\Vert \bar{v}\right\Vert _{p+1}\label{eq:2ConsTh6}\\
 & \leq & -\theta_{2}\left\Vert \bar{v}\right\Vert _{p+1}+\left\Vert \bar{v}\right\Vert _{p+1}\left(-{\textstyle \frac{\gamma_{2}N}{N-1}}\left(\left\Vert \bar{v}\right\Vert _{p+1}\right)^{p}+\varphi_{0}+\theta_{2}\right)\text{, }\theta_{2}>0,\\
 & \leq & -\theta_{2}\left\Vert \bar{v}\right\Vert _{p+1},\,\forall\left\Vert \bar{v}\right\Vert _{p+1}\geq\left({\textstyle \frac{(N-1)\left(\varphi_{0}+\theta_{2}\right)}{\gamma_{2}N}}\right)^{1/p}\label{eq:2ConsTh7}
\end{eqnarray}

The second term on the right side of the inequality \eqref{eq:2ConsTh5.5}
aries from the assumption $\left|\varphi_{i}-\varphi_{j}\right|<\varphi_{0}$.
The relation \eqref{eq:2ConsTh6} is also obtained from the fact that
$\left\Vert \bar{v}\right\Vert \leq\left\Vert \bar{v}\right\Vert _{p+1}$.
\textcolor{black}{According to Lemma \ref{lem:stability} and from the
inequality \eqref{eq:2ConsTh7}, the stability of the perturbed
system \eqref{eq:2ConsTh2} is guaranteed when $\gamma_{2}$ is chosen
large enough.  Then, according to Lemma 5.3 in \cite{khalil1996nonlinear},
practical stationary consensus in finite time is achieved. Thereby,
the proof is complete. }
\end{proof}
%\begin{remark}
%With regards to the fact that in automation engineering the control signal always remains bounded due to technical issues, the assumption $\left|\varphi_{i}-\varphi_{j}\right|<\varphi_{0}, i,j\in\{1,\cdots,N\}$, in Lemma \ref{prop:ConsensusProp2} is not restrictive in practice.
%\end{remark}
{To illustrate that indeed the dynamics \eqref{eq:2AgentsDyn},
when driven under the control law \eqref{eq:2InputLaw}, converges
to the solution of the optimization problem \eqref{eq:1BarrierProb2},
it suffices to verify that the equilibrium point of \eqref{eq:2AgentsDyn}
under the control law \eqref{eq:2InputLaw} coincides with the point
that satisfies the optimality conditions in \eqref{eq:1OptCondition}.
We first solve the distributed convex optimization problem \eqref{eq:1BarrierProb2}
under the condition $g_{i}\left(\cdot\right)=g_{j}\left(\cdot\right)$, $i,j\in\mathcal{N}$.
We then propose a fully distributed algorithm to relax the imposed
condition.}

\subsection{Case I: Agents with Common Constraint}

In this subsection, we prove that under the control law\textcolor{black}{{}
\eqref{eq:2InputLaw}} all agents with dynamics as in \textcolor{black}{\eqref{eq:2InputLaw}}
reach a point that is the solution to the optimization problem \eqref{eq:1BarrierProb2}
when $g_{i}(\cdot)=g(\cdot),\,\forall i\in\mathcal{N}$.
\begin{theorem}
\label{the:Theorem2} Consider Assumptions 1, 2, and 3. If $\frac{\partial^{2}f_{i}(\cdot)}{\partial x_{i}^{2}}=\frac{\partial^{2}f_{j}(\cdot)}{\partial x_{j}^{2}}$
and $g_{i}\left(\cdot\right)=g_{j}\left(\cdot\right)$, $\forall i,j\in\mathcal{N}$,
then, the group of agents with dynamics as in \eqref{eq:2AgentsDyn}
under the control law \eqref{eq:2InputLaw} will converge to the optimum
point of the optimization problem \eqref{eq:1BarrierProb2}.\end{theorem}
\begin{proof}
Consider the Lyapunov function
\begin{equation}
V(\bar{x},\bar{v},t)=\frac{1}{2}\left(\sum_{i=1}^{N}\frac{\partial L_{i}(x_{i},t)}{\partial x_{i}}\right)^{2}+\frac{1}{2}\left(\sum_{i=1}^{N}v_{i}+k_{i}(x_{i},t)\right)^{2}.\label{eq:2The1}
\end{equation}
By calculating the derivative of $V(\bar{x},\bar{v},t)$ with respect
to time along with the trajectories described by the dynamics \eqref{eq:2AgentsDyn},
{it follows that}
\begin{eqnarray}
\dot{V}(\bar{x},\bar{v},t) & = & \sum_{i=1}^{N}\frac{\partial L_{i}(x_{i},t)}{\partial x_{i}}\sum_{i=1}^{N}\left(\frac{\partial^{2}L_{i}(x_{i},t)}{\partial x_{i}^{2}}\dot{x}_{i}+\frac{\partial^{2}L_{i}(x_{i},t)}{\partial x_{i}\partial t}\right)+\nonumber \\
 &  & \left(\sum_{i=1}^{N}v_{i}+k_{i}(x_{i},t)\right)\left(\sum_{i=1}^{N}u_{i}+\frac{dk_{i}(x_{i},t)}{dt}\right)
\end{eqnarray}
From {the conditions} $\frac{\partial^{2}f_{i}(\cdot)}{\partial x_{i}^{2}}=\frac{\partial^{2}f_{j}(\cdot)}{\partial x_{j}^{2}}$
and $g_{i}\left(\cdot\right)=g_{j}\left(\cdot\right)$, $\forall i,j\in\mathcal{N}$,
we can conclude that $\frac{\partial^{2}L_{i}(x_{i},t)}{\partial x_{i}^{2}}=\frac{\partial^{2}L_{j}(x_{j},t)}{\partial x_{j}^{2}}$.
After substituting $u_{i}$ and $k_{i}(x_{i},t)$ in the above equality
with equations \eqref{eq:2InputLaw} and \eqref{eq:k_i}, respectively,
one attains
\begin{eqnarray}
\dot{V}(\bar{x},\bar{v},t) & = & \sum_{i=1}^{N}\frac{\partial L_{i}(x_{i},t)}{\partial x_{i}}\sum_{i=1}^{N}\left(\frac{\partial^{2}L_{i}(x_{i},t)}{\partial x_{i}^{2}}\dot{x}_{i}+\frac{\partial^{2}L_{i}(x_{i},t)}{\partial x_{i}\partial t}\right)+\nonumber \\
 &  & \left(\sum_{i=1}^{N}v_{i}+\left(\frac{\partial^{2}L_{i}(x_{i},t)}{\partial x_{i}^{2}}\right)^{-1}\frac{\partial L_{i}(x_{i},t)}{\partial x_{i}}+\left(\frac{\partial^{2}L_{i}(x_{i},t)}{\partial x_{i}^{2}}\right)^{-1}\frac{\partial^{2}L_{i}(x_{i},t)}{\partial t\partial x_{i}}\right)\nonumber \\
 &  & \left(\sum_{i=1}^{N}-\frac{\partial^{2}L_{i}(x_{i},t)}{\partial x_{i}^{2}}\frac{\partial L_{i}(x_{i},t)}{\partial x_{i}}+r_{i}\right).
\end{eqnarray}
{From Lemma \ref{prop:ConsensusProp2}, one can say that $\sum_{i=1}^{N}r_{i}(t)=0$, $\forall t>\tau_{k}$.
Then, after applying some algebraic simplifications into the above relation, one can verify that}

\begin{eqnarray}
\dot{V}(\bar{x},\bar{v},t) & =- & \left(\sum_{i=1}^{N}\frac{\partial L_{i}(x_{i},t)}{\partial x_{i}}\right)^{2}\label{eq:2The2}\\
 & \leq & 0,\,\,\forall t>\tau_{k}. \label{eq:2The3}
\end{eqnarray}
In the following, we exploit the Barbalat's Lemma \cite{tao1997simple}
to complete the proof.
From the inequality \eqref{eq:2The3}, it can be concluded that $\sum_{i=1}^{N}\frac{\partial L_{i}(x_{i},t)}{\partial x_{i}}\in\mathcal{{L}}^{\infty}$.
By taking integral from $\dot{V}(\bar{x},\bar{v},t)$, one can deduce
that $\sum_{i=1}^{N}\frac{\partial L_{i}(x_{i},t)}{\partial x_{i}}\in\text{\ensuremath{\mathcal{{L}}}}^{2}$.
{We now invoke} Barbalat's lemma \cite{tao1997simple} to conclude that $\underset{t\rightarrow\infty}{\text{lim}}\sum_{i=1}^{N}\frac{\partial L_{i}(x_{i},t)}{\partial x_{i}}=0$. Hence, it follows that

\begin{equation}
\underset{{\scriptstyle t}{\scriptstyle \rightarrow\infty}}{\text{lim}}\sum_{i=1}^{N}\frac{\partial f_{i}(x_{i})}{\partial x_{i}}-\sum_{i=1}^{N}\frac{\alpha}{t+1}\frac{\frac{\partial g_{i}(x_{i})}{\partial x_{i}}}{g(x_{i})}=0.\label{eq:2The4}
\end{equation}
Thereby, the first optimality condition in \eqref{eq:1OptCondition}
is satisfied. In the sequel, we will show that the feasibility condition,
{i.e.} $g({x}_{i}^{*})\leq0$, also holds.
To this end, we employ proof by contradiction. Suppose that we begin
from the initial conditions that satisfy the strict inequality $g\left(x_{i}(0)\right)<0$,
and we have $g\left(x_{i}(t_{k}^{-})\right)<0$ and $g\left(x_{i}(t_{k}^{+})\right)>0$
for some $i$ and finite time $t_{k}>0$. Due to continuity of the function
$g\left(\cdot\right)$, $g(x_{i}(t_{k})$ would be zero. This
implies that $\sum_{i=1}^{N}\frac{\partial L_{i}(x_{i},t)}{\partial x_{i}}$
becomes unbounded at $t_{k}$ which contradicts the fact that $\sum_{i=1}^{N}\frac{\partial L_{i}(x_{i},t)}{\partial x_{i}}\in\mathcal{{L}}^{\infty}$
as shown above. Therefore, the inequality $g(x_{i}(t))<0$ with
$g\left(x_{i}(0)\right)<0$ holds for $t>0$, and this ends the
proof.
\end{proof}

In the next subsection, we will propose an algorithm similar to the
one presented in Subsection 3.2. The proposed algorithm {enables
us to relax} the requirement $g_{i}(\cdot)=g_{j}(\cdot),\,\forall i,j\in\{1,\ldots,N\}$.

\subsection*{4.2 \label{sub:4.2-Case-II}Case II: Agents with Distinct Constraints}

In this subsection, we utilize {the same} distributed average tracking tool
as the one in Subsection 3.2. We then illustrate how all agents with
dynamics as in \eqref{eq:2AgentsDyn} converge to the solution of the optimization
problem \eqref{eq:1BarrierProb2} and reach consensus on their first
states, i.e. their positions, when agents admit distinct constraints.

Consider the double-integrator dynamics
\begin{equation}
\ensuremath{\begin{array}{l}
{\dot{x}}(t)={v}(t),\\
{\dot{v}}(t)={u}(t).
\end{array}}\label{eq:2Cent1}
\end{equation}
with the {strictly convex and twice differentiable} objective function $Q(x,t)$.
\begin{lemma}
\label{CentCont}The following control input drives \textcolor{black}{the
dynamics stated by} \eqref{eq:2Cent1} to the minimizer of the strictly convex
objective function \textup{$Q(x,t)$, }
\end{lemma}
\begin{equation}
u(t)=-\frac{d}{dt}\left(\left(\frac{\partial^{2}Q}{\partial x^{2}}\right)^{-1}\frac{\partial Q}{\partial x}+\left(\frac{\partial^{2}Q}{\partial x^{2}}\right)^{-1}\frac{\partial^{2}Q}{\partial t\partial x}\right)-\frac{\partial^{2}Q}{\partial x^{2}}\frac{\partial Q}{\partial x}.\label{eq:2Cent2}
\end{equation}

\begin{proof}
{We start by defining the following Lyapunov function
}as
\begin{equation}
V(x,t)=\frac{1}{2}\left(\frac{\partial Q}{\partial x}\right)^{2}+\frac{1}{2}\left(v+\left(\frac{\partial^{2}Q}{\partial x^{2}}\right)^{-1}\frac{\partial Q}{\partial x}+\left(\frac{\partial^{2}Q}{\partial x^{2}}\right)^{-1}\frac{\partial^{2}Q}{\partial t\partial x}\right)^{2}.\label{eq:5Lemma1:1}
\end{equation}
One can calculate the time derivative of \eqref{eq:5Lemma1:1} {along
the trajectories of the dynamics} $\eqref{eq:2Cent1}$
and obtain
\begin{eqnarray*}
\dot{V}(x,t) & = & \frac{\partial Q}{\partial x}\left(\frac{\partial^{2}Q}{\partial x^{2}}\dot{x}+\frac{\partial^{2}Q}{\partial x\partial t}\right)+\left(v+\left(\frac{\partial^{2}Q}{\partial x^{2}}\right)^{-1}\frac{\partial Q}{\partial x}+\left(\frac{\partial^{2}Q}{\partial x^{2}}\right)^{-1}\frac{\partial^{2}Q}{\partial t\partial x}\right)\\
 &  & \left(u+\frac{d}{dt}\left(\left(\frac{\partial^{2}Q}{\partial x^{2}}\right)^{-1}\frac{\partial Q}{\partial x}+\left(\frac{\partial^{2}Q}{\partial x^{2}}\right)^{-1}\frac{\partial^{2}Q}{\partial t\partial x}\right)\right).
\end{eqnarray*}
By substituting $u$ in $\dot{V}(x,t)$ with \eqref{eq:2Cent2}, it is easy to verify that
\begin{eqnarray}
\dot{V}(x,t) & = & -\left(\frac{\partial Q}{\partial x}\right)^{2}\label{eq:5Lemma1:2}\\
 & \leq & 0\label{eq:5Lemma1:3}
\end{eqnarray}
Now, we apply Barbalat's Lemma \cite{tao1997simple} to clear the
proof. Due to passivity of $V(x,t)$, one can derive that $V(x,t)\in\mathcal{{L}}^{2}$.
Moreover, from \eqref{eq:5Lemma1:3}, $V(x,t)\in\mathcal{{L}^{\infty}}$
holds. {Hence}, $V(x,t)\rightarrow0$ as $t\rightarrow\infty$.
It implies that ${\displaystyle \underset{t\rightarrow\infty}{lim}\frac{\partial Q}{\partial x}}=0$.
{This concludes the proof.}
\end{proof}

{We now} exploit the result of Lemma \ref{CentCont}
to minimize a collective convex objective function {in
a distributed fashion}. {We propose that the control
law}
\begin{eqnarray}
u_{i}(t) & = & -\frac{d}{dt}\left(\left(\frac{\partial^{2}\sum_{i=1}^{N}L_{i}}{\partial x_{i}^{2}}\right)^{-1}\frac{\partial\sum_{i=1}^{N}L_{i}}{\partial x_{i}}+\left(\frac{\partial^{2}\sum_{i=1}^{N}L_{i}}{\partial x_{i}^{2}}\right)^{-1}\frac{\partial^{2}\sum_{i=1}^{N}L_{i}}{\partial t\partial x_{i}}\right)\nonumber \\
 &  & -\frac{\partial^{2}\sum_{i=1}^{N}L_{i}}{\partial x_{i}^{2}}\frac{\partial\sum_{i=1}^{N}L_{i}}{\partial x_{i}},\label{eq:2Cent3}
\end{eqnarray}
solves the convex optimization problem \eqref{eq:1BarrierProb2}.
{However, it requires computation of the terms that are not available
to $i$-th agent. We exploit a distributed average tracking
tool that enables each agent to estimate these terms in a cooperative
fashion.}

Consider the agents \eqref{eq:2AgentsDyn} under {the
graph} $\mathcal{{G}}$. Suppose that each agent \textcolor{black}{admits}
the following dynamics

\begin{align}
\dot{\zeta}_{i}(t) & =-a\sum_{j\in\mathcal{{N}}_{i}}sgn(\chi_{i}(t)-\chi_{j}(t)),\label{eq:IntDynamics1}\\
\dot{\xi_{i}}(t) & =-b\sum_{j\in\mathcal{N}{}_{i}}sgn(\mu_{i}(t)-\mu_{j}(t)),\label{eq:2IntDynamics2}
\end{align}
where
\begin{equation}
\chi_{i}(t)=\zeta_{i}(t)+\left[{\displaystyle \begin{array}{c}
\frac{\partial L_{i}(x_{i},t)}{\partial x_{i}}\\
\frac{\partial^{2}L_{i}(x_{i},t)}{\partial x_{i}\partial t}\\
\frac{d}{dt}\frac{\partial L_{i}(x_{i},t)}{\partial x_{i}}\\
\frac{d}{dt}\frac{\partial^{2}L_{i}(x_{i},t)}{\partial x_{i}\partial t}
\end{array}}\right],\label{eq:2IntSignal1}
\end{equation}
and
\begin{align}
\mu_{i}(t) & =\xi_{i}(t)+\left[\begin{array}{c}
\frac{\partial^{2}L_{i}(x_{i},t)}{\partial x_{i}^{2}}\\
\frac{d}{dt}\frac{\partial^{2}L_{i}(x_{i},t)}{\partial x_{i}^{2}}
\end{array}\right].\label{eq:2IntSignal2}
\end{align}

{It is easy to see that over the graph} $\mathcal{G}$,
$\sum_{i=1}^{N}\dot{\zeta}_{i}(t)=0$. If we assume that $\sum_{i=1}^{N}\zeta(0)=0$,
then, it concludes $\sum_{i=1}^{N}\zeta_{i}(t)=0$ for $t>0$
since $\sum_{i=1}^{N}\dot{\zeta}_{i}(t)=0$. Now, from
the equation \eqref{eq:2IntSignal1}, we have $\sum_{i=1}^{N}\chi_{i}(t)=\sum_{i=1}^{N}\left[{\displaystyle \begin{array}{c}
\frac{\partial L_{i}(x_{i},t)}{\partial x_{i}}\\
\frac{\partial^{2}L_{i}(x_{i},t)}{\partial x_{i}\partial t}\\
\frac{d}{dt}\frac{\partial L_{i}(x_{i},t)}{\partial x_{i}}\\
\frac{d}{dt}\frac{\partial^{2}L_{i}(x_{i},t)}{\partial x_{i}\partial t}
\end{array}}\right]$. It follows from Theorem 1 in \cite{chen2012distributed} that with
$a>\underset{t}{sup}\left\{ \left\Vert \zeta_{i}(x_{i},t)\right\Vert _{\infty}\right\} $,
$\forall i\in\mathcal{N}$, $\left|\chi_{i}(t)-\chi_{j}(t)\right|=0\,\forall i,j\in\mathcal{N}$
in an upper-bounded finite time, say  $T_{k}$. With $\chi_{i}(t)=\chi_{j}(t)$,
for $t>T_{k}$, the following holds,
\begin{equation}
\chi_{i}(t)=\frac{1}{N}\sum_{i=1}^{N}\left[{\displaystyle \begin{array}{c}
\frac{\partial L_{i}(x_{i},t)}{\partial x_{i}}\\
\frac{\partial^{2}L_{i}(x_{i},t)}{\partial x_{i}\partial t}\\
\frac{d}{dt}\frac{\partial L_{i}(x_{i},t)}{\partial x_{i}}\\
\frac{d}{dt}\frac{\partial^{2}L_{i}(x_{i},t)}{\partial x_{i}\partial t}
\end{array}}\right].\label{eq:2est1}
\end{equation}
{Following a same line of reasoning, it can be concluded that after
a finite time}
\begin{eqnarray}
\mu_{i}(t) & = & \frac{1}{N}\sum_{i=1}^{N}\left[\begin{array}{c}
\frac{\partial^{2}L_{i}(x_{i},t)}{\partial x_{i}^{2}}\\
\frac{d}{dt}\frac{\partial^{2}L_{i}(x_{i},t)}{\partial x_{i}^{2}}
\end{array}\right],\label{eq:2est2}
\end{eqnarray}
where $b>\underset{t}{sup}\left\{ \left\Vert \mu_{i}(x_{i},t)\right\Vert _{\infty}\right\}, \forall i\in\mathcal{N}$.

Now, we present the new protocol
\begin{align}
u_{i} & =\mu_{i1}^{-2}\mu_{i2}(\text{\ensuremath{\chi}}_{i1}+\chi_{i2})-\mu_{i1}^{-1}(\chi_{i3}+\chi_{i4})-\mu_{i1}\chi_{i1}+r_{i},\label{eq:2InputLaw2}
\end{align}
{where $\chi_{i1}=\frac{\partial L_{i}(x_{i},t)}{\partial x_{i}}$,
$\chi_{i2}=\frac{\partial^{2}L_{i}(x_{i},t)}{\partial x_{i}\partial t}$,
$\chi_{i3}=\frac{d}{dt}\frac{\partial L_{i}(x_{i},t)}{\partial x_{i}}$,
$\chi_{i4}=\frac{d}{dt}\frac{\partial^{2}L_{i}(x_{i},t)}{\partial x_{i}\partial t}$,
$\mu_{i1}=\frac{\partial^{2}L_{i}(x_{i},t)}{\partial x_{i}^{2}}$,
and $\mu_{i2}=\frac{d}{dt}\frac{\partial^{2}L_{i}(x_{i},t)}{\partial x_{i}^{2}}$.
}To prove consensus on the position states, we refer the readers to
Lemma \ref{prop:ConsensusProp2}.
\begin{lemma}
Suppose that Assumptions 1, 2, and 3 hold and $\sum_{i=1}^{N}\zeta_{i}(0)=0$,
$\sum_{i=1}^{N}\xi_{i}(0)=0$, $a>\underset{t}{sup}\left\{ \left\Vert \zeta_{i}(x_{i},t)\right\Vert _{\infty}\right\} $,
and $b>\underset{t}{sup}\left\{ \left\Vert \mu_{i}(x_{i},t)\right\Vert _{\infty}\right\} ,\,\forall i\in\mathcal{N}$.
Then, the protocol \eqref{eq:2InputLaw2} {drives}
the agents {with dynamics as in} \eqref{eq:2AgentsDyn}
to the solution of \eqref{eq:1BarrierProb2}. \end{lemma}
\begin{proof}
Let us define the following Lyapunov candidate function,
\[
V(\bar{\chi},\bar{\mu},t)=\frac{1}{2}\left(\sum_{i=1}^{N}\chi_{i1}\right)^{2}+\frac{1}{2}\left(\sum_{i=1}^{N}v_{i}+\mu_{i1}^{-1}\left(\chi_{i1}+\chi_{i2}\right)\right)^{2},
\]
where $\bar{\chi}=\left[\chi_{1}^{\top},\ldots,\chi_{N}^{\top}\right]^{\top}$
and $\bar{\mu}=\left[\mu_{1}^{\top},\ldots,\mu_{N}^{\top}\right]^{\top}$.
After calculating the time derivative of $V(\bar{\chi},\bar{\mu},t)$
along the trajectories of \eqref{eq:2AgentsDyn}, the following holds,
\begin{align}
\dot{V}(\bar{\chi},\bar{\mu},t) & =\sum_{i=1}^{N}\chi_{i1}\sum_{i=1}^{N}\chi_{i3}+\left(\sum_{i=1}^{N}v_{i}+\mu_{i1}^{-1}\left(\chi_{i1}+\chi_{i2}\right)\right)\nonumber \\
 & \left(\sum_{i=1}^{N}u_{i}-\mu_{i1}^{-2}\mu_{i2}(\text{\ensuremath{\chi}}_{i1}+\chi_{i2})+\mu_{i1}^{-1}(\chi_{i3}+\chi_{i4})\right).\label{eq:Lemma6:1}
\end{align}
From \eqref{eq:2InputLaw2} and \eqref{eq:Lemma6:1}, we can write
\begin{eqnarray*}
\dot{V}(\bar{\chi},\bar{\mu},t) & = & \sum_{i=1}^{N}\chi_{i1}\sum_{i=1}^{N}\chi_{i3}+\left(\sum_{i=1}^{N}v_{i}+\mu_{i1}^{-1}\left(\chi_{i1}+\chi_{i2}\right)\right)\left(-\sum_{i=1}^{N}\mu_{i1}\chi_{i1}\right),
\end{eqnarray*}
in which we have used from the fact that $\sum_{i=1}^{N}r_{i}(t)=0$
in finite time, i.e. $t>\tau_{k}$ according to Lemma \ref{prop:ConsensusProp2}.
Note that $\sum_{i=1}^{N}\chi_{i3}=\sum_{i=1}^{N}\mu_{i1}v_{i}+\chi_{i2}$.
{According to} Theorem 1 in \cite{chen2012distributed},
 there exists a finite time $T_{k}$ after which $\mu_{i}(t)=\mu_{j}(t)$,
$\chi_{i}(t)=\chi_{j}(t)$, $\forall i,j\in\mathcal{N}$, {when}
$\sum_{i=t}^{N}\zeta_{i}(0)=0$, $\sum_{i=t}^{N}\xi_{i}(0)=0$, $a>\underset{t}{sup}\left\{ \left\Vert \zeta_{i}(x_{i},t)\right\Vert _{\infty}\right\} ,\,\forall i\in\{1,\ldots,N\}$,
and $b>\underset{t}{sup}\left\{ \left\Vert \mu_{i}(x_{i},t)\right\Vert _{\infty}\right\} \forall i.$
Thus,
\begin{eqnarray}
\dot{V}(\bar{\chi},\bar{\mu},t) & = & -\left(\sum_{i=1}^{N}\chi_{i1}\right)^{2}\label{eq:EstLyapDerFin}\\
 & \leq & 0,\,\forall t>T_{k}+\tau_{k}.\label{eq:EstLyapDerFin1}
\end{eqnarray}
We assert that the solutions to locally Lipschitz dynamics \eqref{eq:2AgentsDyn}
with a bounded input stay bounded in a finite time. Thus, for $t<T_{k}+\tau_{k}$,
$\chi_{i1}$, $\forall i\in\mathcal{N}$, remains bounded. From
\eqref{eq:EstLyapDerFin1} and by means of Barbalat's lemma, one can
show that $\sum_{i=1}^{N}\chi_{i1}=0$ as $t\rightarrow\infty$.
Thus, the stationary optimality condition is achieved. {The
remainder of the proof {for the feasibility condition} can be done similar to that of Theorem \ref{the:Theorem2};
hence, it is omitted here.}
\end{proof}

\section{Numerical Simulation}

{\crb This section provides numerical simulation to demonstrate the performance
of the presented distributed algorithm. We consider eight {2-dimensional}
double-integrator agents that move in a 2-D plane with $x$ and $y$
axis. In our simulation, the information sharing graph $\mathcal{{G}}$
is set as: $1\Leftrightarrow2\Leftrightarrow3\Leftrightarrow4\Leftrightarrow5\Leftrightarrow6\Leftrightarrow7\Leftrightarrow8$. Set
the initial conditions for the positions of agents $1,\ldots,\,\text{and}\,8$ as $\left(-1,1\right),\left(1,0\right), \left(2,-1\right)$,
 $\left(2,2\right)$, $\left(0,4\right)$, $\left(1.5,1.5\right)$, $\left(0.5,4\right)$, and, $\left(1.5,0\right)$, respectively. The
local objective functions for all six agents are as follows:}
\begin{eqnarray*}
f_{1}(x_{1},y_{1}) & = & (x_{1}-4)^{2}+(y_{1}-5)^{2},\\
f_{2}(x_{2},y_{2}) & = & (x_{2}-5)^{2}+(y_{2}-9)^{2},\\
f_{3}(x_{3},y_{3}) & = & (x_{3}-4)^{2}+(y_{3}-10)^{2},\\
f_{4}(x_{4},y_{4}) & = & (x_{4}-3)^{2}+(y_{4}+6)^{2},\\
f_{5}(x_{5},y_{5}) & = & (x_5+1)^2+y_{5}^2,\\
f_{6}(x_{6},y_{6}) & = & (x_6+3)^2+(y_6+3)^2,\\
f_{7}(x_{7},y_{7}) & = & (x_7-4)^2+(y_7-1)^2,\\
f_{8}(x_{8},y_{8}) & = & x_8^2+(y_8-8)^2.
\end{eqnarray*}

{\crb Agent 1 admits the inequality constraint $g_{1}(x_{1},y_{1})=x_{1}+y_{1}-5\leq0$.
Agent 2 has the local  constraint $g_{2}(x_{2},y_{2})=x_{2}^{2}+y_{2}^{2}-10\leq0$.
Agent 3 adopts has the local  constraint $g_{3}(x_{3},y_{3})=x_{3}^{2}+y_{3}^{2}-10\leq0$.
Agent 4 accepts the constraint of $g_{4}(x_{4},y_{4})=x_{4}+y_{4}-12\leq0$ while
agents 5 and 6 are subject to the constraints $g_5(x_5)=x_5-2\leq0$ and $g_6(x_6,y_6)=x_6+y_6-4\leq0$, respectively.
The agents 7 is restricted to the constraint $g_7(x_7)=(x_7-3)^2-4\leq0$, and, finally, the agent 8's movement along the y-axis is
constrained by the inequality $g_8(y_8)=(y_8-2)^2-9 \leq 0$.
The global optimum point is $(1,3)$.} We adopt the protocol \eqref{eq:2InputLaw2} to drive all agents
toward the optimal consensus point. Suppose that each agent has an
internal dynamics as in \eqref{eq:IntDynamics1} and \eqref{eq:2IntDynamics2}
to construct the control protocol \eqref{eq:2InputLaw2}, where
we choose $a=20$ and $b=20$. The trajectories of all
agents are shown in Figure 1. In Figure 2, trajectories of the agents'
velocities along $x$ and $y$ dimensions are plotted.

\begin{figure}
\label{fig:figone}
\begin{center}
\resizebox*{10cm}{!}{\includegraphics{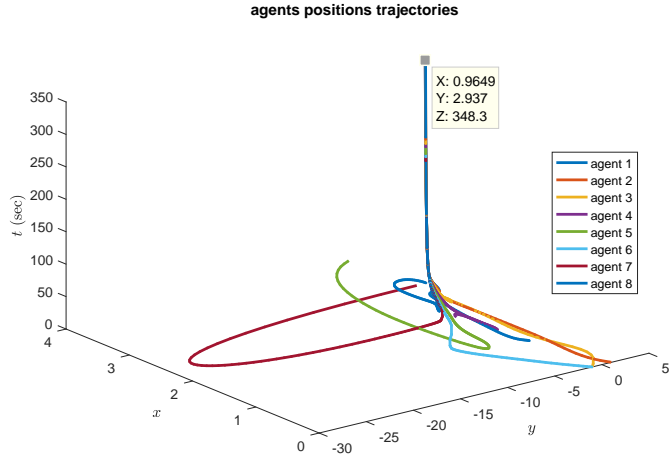}}\hspace{5pt}
\caption{Position trajectories of double-integrator agents in numerical example.}
\end{center}
\end{figure}

\begin{figure}
\begin{center}
\resizebox*{10cm}{!}{\includegraphics{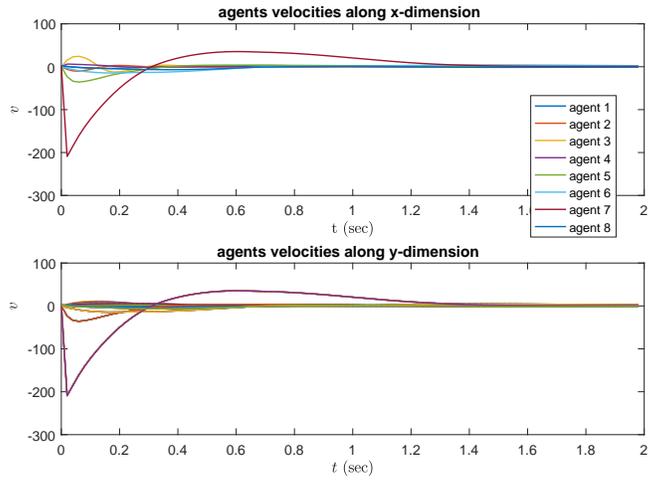}}\hspace{5pt}
\label{fig:figtwo}
\caption{Velocity trajectories of double-integrator agents along both $x$ (down)
and $y$ (top) dimension.}
\end{center}

\end{figure}

\section{Conclusion}

The problem of distributed constrained optimal consensus for undirected
networks of dynamical agents was fully investigated in this paper. Here,
all agents are supposed to rendezvous  at a point that
minimizes a collective convex objective function with regard to some local
constraints. We studied this problem for two typical dynamics, namely
single-integrator and double-integrator dynamics. To tackle the
problem, we split it into two separate subproblems, viz consensus subproblem and
distributed constrained convex optimization one. Then, we proposed a distributed
control law composed of a consensus protocol and a term associated with decentralized convex
optimization algorithm. In the proposed setup, each agent requires
to know of its own states and the relative positions of agents within
its neighborhood set. No information associated with objective functions
are exchanged between agents.

To certify consensus, we exploited some theory associated with the analysis of perturbed
systems stability. As for constrained convex optimization
algorithm, we adopted an extended form of the interior-point method. Then, through Barbalat's lemma, it was illustrated that optimality conditions,
including the stationary condition and the feasibility condition,
uniformly hold.

Finally, to relax the restricting assumption of \textcolor{black}{local constraints being common}, we exploited the distributed average tracking
tool to estimate some essential information associated with the whole network at the local level.
Then, we proved the convergence of our algorithm.

\section{Appendix}

In this appendix, we clarify how the inequality $\bar{v}^{\top}\Pi sig(\bar{v})^{p}\geq{\scriptstyle {\textstyle \frac{N-1}{N}}}\left(\left\Vert \bar{v}\right\Vert _{p+1}\right)^{p+1}$
holds for any $\bar{v}\in\mathbb{{R}}^{N}$.
\begin{align}
\bar{v}^{\top}\Pi sig(\bar{v})^{p} & = \bar{v}^{\top}\left[\begin{array}{cccc}
\frac{N-1}{N} & -\frac{1}{N} & \cdots & -\frac{1}{N}\\
-\frac{1}{N} & \frac{N-1}{N} & \cdots & -\frac{1}{N}\\
\vdots & \vdots & \ddots & \vdots\\
-\frac{1}{N} & -\frac{1}{N} & \cdots & \frac{N-1}{N}
\end{array}\right]sig(\bar{v})^{p}\nonumber \\
 & =  {\scriptstyle {\textstyle \frac{N-1}{N}}}\sum_{i=1}^{N}\left|v_{i}\right|^{p+1}-{\textstyle {\textstyle \frac{1}{N}}}\sum_{i=1}^{N}\sum_{\underset{j\neq i}{j=1}}^{N}v_{j}\left|v_{i}\right|^{p}sign(v_{i})\\
 & \geq {\scriptstyle {\textstyle \frac{N-1}{N}}}\sum_{i=1}^{N}\left|v_{i}\right|^{p+1}-{\textstyle {\textstyle \frac{1}{N}}}\sum_{i=1}^{N}\sum_{\underset{j\neq i}{j=1}}^{N}\left|v_{j}\right|\left|v_{i}\right|^{p}\label{eq:A3}\\
 & =  {\textstyle \frac{1}{N}\left|v_{1}\right|^{p}}\sum_{j=2}^{N}\left(\left|v_{1}\right|-\left|v_{j}\right|\right)\nonumber\\
 & +  {\textstyle \frac{1}{N}\left|v_{2}\right|^{p}}\sum_{\underset{j\neq2}{j=1}}^{N}\left(\left|v_{2}\right|-\left|v_{j}\right|\right)+\cdots+{\textstyle \frac{1}{N}\left|v_{N}\right|^{p}}\sum_{\underset{j\neq N}{j=1}}^{N}\left(\left|v_{N}\right|-\left|v_{j}\right|\right)\label{eq:A1}\\
 & =  I
\end{align}
Now, assume that $\left|v_{i}\right|-\left|v_{j}\right|=\epsilon_{ij}$
for some $\epsilon_{ij}\in\mathbb{{R}}$. Then, we rewrite the left
side of the equality \eqref{eq:A1} as

\begin{equation}
I=\frac{1}{2}\left({\textstyle \frac{1}{N}}\sum_{j=2}^{N}\epsilon_{1j}\left(\left|v_{1}\right|^{p}-\left|v_{j}\right|^{p}\right)+\cdots+{\textstyle \frac{1}{N}}\sum_{\underset{j\neq N}{j=1}}^{N}\epsilon_{Nj}\left(\left|v_{N}\right|^{p}-\left|v_{j}\right|^{p}\right)\right).\label{eq:A2}
\end{equation}
It is straightforward to show that $I\geq0$, for any $\bar{v}\in\mathbb{{R}}^{N}$.
Therefore, $\bar{v}^{\top}\Pi sig(\bar{v})^{p}\geq0,\forall\bar{v}\in\mathbb{R}^{N}$.
On the other hand, from \eqref{eq:A3}, we have
\begin{eqnarray}
\bar{v}^{\top}\Pi sig(\bar{v})^{p} & \geq & {\scriptstyle {\textstyle \frac{N-1}{N}}}\sum_{i=1}^{N}\left|v_{i}\right|^{p+1}\nonumber \\
 & = & {\scriptstyle {\textstyle \frac{N-1}{N}}}\left(\left\Vert \bar{v}\right\Vert _{p+1}\right)^{p+1}.\label{eq:A4}
\end{eqnarray}

\bibliographystyle{apacite}
\bibliography{interactapasample}

\end{document}